\documentclass[12pt]{amsart}
\usepackage{amssymb,amsmath,amsthm,amscd}
\usepackage[all]{xy}

\numberwithin{equation}{section}

\newtheoremstyle{fancy1}{10pt}{10pt}{\itshape}{12pt}{\textsc\bgroup}{.\egroup}{8pt}{
}
\newtheoremstyle{fancy2}{10pt}{10pt}{}{12pt}{\itshape}{.}{8pt}{ }

\theoremstyle{fancy1}

\newtheorem{cor}[equation]{Corollary}
\newtheorem{lem}[equation]{Lemma}
\newtheorem{prop}[equation]{Proposition}
\newtheorem{thm}[equation]{Theorem}

\newtheorem{main}{Theorem}
\newtheorem*{main*}{Theorem}

\newtheorem*{cor*}{Corollary}

\newtheorem*{problem*}{Problem}

\renewcommand{\thetable}{\theequation}
\theoremstyle{fancy2}

\newtheorem*{rem*}{Remark}

\newcommand{\cref}[1]{Corollary~\ref{#1}}

\newcommand{\lref}[1]{Lemma~\ref{#1}}
\newcommand{\pref}[1]{Proposition~\ref{#1}}

\newcommand{\tref}[1]{Theorem~\ref{#1}}

\newcommand{\gt}{\theta}

\newcommand{\gs}{\sigma}

\newcommand{\RP}{\mathbb{R\mkern1mu P}}
\newcommand{\CP}{\mathbb{C\mkern1mu P}}

\newcommand{\Sph}{\mathbb{S}}
\newcommand{\Sphpm}{\mathbb{S}^{\scriptscriptstyle{\pm}}}

\newcommand{\C}{{\mathbb{C}}}
\newcommand{\R}{{\mathbb{R}}}
\newcommand{\Z}{{\mathbb{Z}}}
\newcommand{\QH}{{\mathbb{H}}}

\renewcommand{\H}{\ensuremath{\operatorname{H}}}

\newcommand{\G}{\ensuremath{\operatorname{G}}}

\newcommand{\SO}{\ensuremath{\operatorname{SO}}}

\renewcommand{\O}{\ensuremath{\operatorname{O}}}

\newcommand{\U}{\ensuremath{\operatorname{U}}}
\newcommand{\SU}{\ensuremath{\operatorname{SU}}}

\newcommand{\T}{\ensuremath{\operatorname{T}}}
\renewcommand{\S}{\ensuremath{\operatorname{S}}}
\renewcommand{\P}{\ensuremath{\operatorname{P}}}

\newcommand{\N}{\ensuremath{\operatorname{N}}}
\newcommand{\K}{\ensuremath{\operatorname{K}}}
\renewcommand{\L}{\ensuremath{\operatorname{L}}}

\newcommand{\fso}{{\mathfrak{so}}}

\newcommand{\orth}[2]{\text{S(O(#1)O(#2))}}
\def\con#1=#2(#3){#1 \equiv #2 \bmod{#3}}

\newcommand{\diag}{\ensuremath{\operatorname{diag}}}

\newcommand{\rank}{\ensuremath{\operatorname{rk}}}

\DeclareMathOperator{\Hom}{Hom} \DeclareMathOperator{\id}{id}
\DeclareMathOperator{\Id}{Id}

\newcommand{\Kpo}{\K_{\scriptscriptstyle{0}}^{\scriptscriptstyle{+}}}
\newcommand{\Kmo}{\K_{\scriptscriptstyle{0}}^{\scriptscriptstyle{-}}}
\newcommand{\Kpm}{\K^{\scriptscriptstyle{\pm}}}
\newcommand{\Kp}{\K^{\scriptscriptstyle{+}}}
\newcommand{\Km}{\K^{\scriptscriptstyle{-}}}

\newcommand{\Dp}{\mathbb{D}^{\scriptscriptstyle{+}}}
\newcommand{\Dm}{\mathbb{D}^{\scriptscriptstyle{-}}}
\newcommand{\Dpm}{\mathbb{D}^{\scriptscriptstyle{\pm}}}

\newcommand{\subo}{_{\scriptscriptstyle{0}}}

\newcommand{\no}{\noindent}
\newcommand{\co}{{cohomogeneity}}
\newcommand{\coo}{{cohomogeneity one}}
\newcommand{\com}{{cohomogeneity one manifold}}

\newcommand{\coa}{{cohomogeneity one action}}

\newcommand{\nn}{non-negative}
\newcommand{\nnc}{non-negative curvature}

\newcommand{\pont}{Pontryagin}

\begin{document}

\title{Lifting Group Actions and Nonnegative Curvature}

\author{Karsten Grove}
\address{University of Maryland\\
     College Park , MD 20742}
\email{kng@math.umd.edu}
\author{Wolfgang Ziller}
\address{University of Pennsylvania\\
     Philadelphia, PA 19104}
\email{wziller@math.upenn.edu}
\thanks{The first named author was supported in part by the Danish Research
Council and the second author by the Francis J. Carey Term Chair and
the Clay Institute. Both authors were supported by grants from the
National Science Foundation.}

\maketitle



Since the emergence of the fundamental structure theorem for
complete open manifolds of nonnegative curvature due to Cheeger and
Gromoll \cite{CG}, one of the central issues in this area has been
to what extent the converse to this so-called soul theorem holds. In
other words: Which total spaces of vector bundles over compact
nonnegatively curved manifolds  admit (complete) metrics with
nonnegative curvature? The first examples where no such metrics
exist were found by Ozaidin and Walschap \cite{OW}. More recently a
wealth of other examples have been found by Belegradek and Kapovitch
\cite{BK1},\cite{BK2}. So far, however, no obstructions are known
when the base has finite fundamental group, and in particular when
it is simply connected.

It is well known and easy to see that all vector bundles over
compact simply connected 2- and 3-manifolds with nonnegative
curvature (i.e. $\Sph^2$ and $\Sph^3$ by \cite{Ha}) admit complete
metrics with nonnegative curvature. The first non-trivial case was
treated in \cite{GZ}, where it was shown
    that all vector bundles over $\Sph^4$ can be equipped with
such a metric. In this paper,  we consider vector bundles over the
remaining known closed simply connected 4-manifolds with nonnegative
curvature, i.e., $\CP^2$ , $\Sph^2\times \Sph^2$, and $\CP^2 \#
{\pm}\CP^2$. One easily sees that a vector bundle over any of these
spaces admits a complete metric of nonnegative curvature if its
structure group reduces to a torus, see \tref{general}.
 It turns out that most vector bundles over $\Sph^2\times \Sph^2$,
 and $\CP^2 \#
{\pm}\CP^2$ are of this form, see \tref{product-sum}, in contrast to
vector bundles over $\CP^2$.  As a consequence of one of our main
results we have:

\begin{main}
The total space of every vector bundle over $\CP^2$ with non-trivial
second Stiefel Whitney class $w_2$ admits a complete metric of
nonnegative sectional curvature.
\end{main}

    In the case of $w_2=0$, our methods will show that
     half of all 3-dimensional vector bundles,
    those with $p_1 \equiv 0 \text{ mod }
    8$, admit non-negative curvature, and they all do when the fiber dimension is at least $5$.
    We do not know if the remaining bundles with $w_2=0$ admit
    nonnegative curvature, although some of them do since their structure
    group reduces to a torus.

In the more special case of complex vector bundles, we will see
that:

\begin{main}
The total space of any complex rank 2 vector bundle over $\CP^2$
admits a complete metric of nonnegative curvature if its first Chern
class $c_1$ is odd. The same is true if $c_1$ is even and the
discriminant $ \Delta = c_1^2-4c_2$ satisfies $\Delta \equiv 0
\text{ mod }8$.
\end{main}

When $c_1$ is even, this is again half of all possible complex
vector bundles since in general $\Delta \equiv 0 \text{ mod }4 $.
Although we do not know an explicit connection, it is tantalizing to
observe that in the classification of (stable) holomorphic vector
bundles over $\CP^2$, see \cite{OSS}, one also has the same division
into $c_1$ even and odd ($c_1$ odd being easier) and that the parity
of $\Delta/4$ in the case $c_1$ even is equally important.

\smallskip

In \cite{Ri}, it was shown that stably every vector bundle over
$\Sph^n$ admits a complete metric of nonnegative curvature. Our
analysis  yields the same claim for vector bundles  over each of
$\CP^2$ , $\Sph^2\times \Sph^2$, and $\CP^2 \# {-}\CP^2$, in fact
they admit such a metric as long as the fiber dimension is at least
 six.

\smallskip

Our results rely on constructing invariant
metrics of nonnegative curvature on principal bundles, and then get
the desired metrics on the associated bundles from the well known
curvature increasing property of Riemannian submersions. This of
course also implies that the associated sphere bundles over $\CP^2$
admit nonnegative curvature as well, giving rise to an interesting
new class of compact examples with nonnegative curvature.

\smallskip

In general it is a very difficult problem to decide which principal
bundles over nonnegatively curved manifolds admit metrics with
nonnegative curvature.  A general construction of principal bundles
over manifolds of cohomogeneity one, i.e. $\G$-manifolds with
one-dimensional orbit space, was found in \cite{GZ} (cf. section 1).
There it was also shown that a large class of \coo\ manifolds, the
ones where the singular orbits have codimension two, carry metrics
of nonnegative curvature, giving rise to such metrics on all
principal bundles over $\Sph^4$.

\smallskip

Our point of departure here is that, as in the case of $\Sph^4$,
each of the manifolds  $\CP^2$ , $\Sph^2\times \Sph^2$, and $\CP^2
\# - \CP^2$ support cohomogeneity one actions with singular orbits
of codimension two. Therefore so do all the principal bundles
constructed by the \emph{cohomogeneity one method} alluded to above.
It remains to determine which bundles one gets this way, a
topological problem which is considerably more involved than the
corresponding one for bundles over $\Sph^4$ solved in \cite{GZ}. One
can formulate this problem in purely topological terms as follows:

\begin{problem*}
Given a principal $\L$-bundle $P \to M$ over a $\G$-manifold $M$.
When does  the action of $\G$ on $M$ lift to an action of $\G$, or
possibly a cover of $\G$, on the total space $P$, such that the lift
commutes with $L$.
\end{problem*}

We will refer to such a lift as a \emph{commuting lift}. This
problem has been studied extensively, see e.g. \cite{HH}, \cite{HY},
\cite{La}, \cite{PS}, \cite{St}, \cite{TD} and references therein.
However, apart from the general result \cite{PS} that every action
of a semi simple group admits a commuting lift to the total space of
every principal circle or more generally torus bundle, the results
seem to be difficult to apply in concrete cases.

With this terminology  we showed in \cite{GZ} that the cohomogeneity one action of $\SO(3)$
on $\Sph^4$ admits a commuting lift to every principal $\SO(k)$
bundle over $\Sph^4$. In contrast, we will show that the
cohomogeneity one action of $\SO(3)$ on $\CP^2$ does not lift to
every principal bundle over $\CP^2$, giving rise to the exceptions
in Theorem A and B. More precisely, in this language, the
topological main result behind Theorem A and B can be formulated as:

\begin{main}
Let $P \to \CP^2$ be a principal $\SO(3)$-bundle. The cohomogeneity
one $\SO(3)$ subaction of the standard $\SU(3)$-action on $\CP^2$
admits a commuting lift to $P$ if and only if one of the following
holds:

\begin{itemize}
\item[a)]The principal bundle is not spin, i.e. $w_2(P)\ne 0$.
\item[b)] The bundle is spin and $p_1(P) \equiv 0 \text{ mod } 8$.
\item[c)] The Pontryagin class satisfies $p_1(P)=4r^2 $ for some integer $r>0$.
\end{itemize}
\end{main}

\smallskip

\no More generally, our methods address the question which principal
$\SO(k)$ bundles over each simply connected cohomogeneity one
4-manifold  admit a commuting lift. It will follow, e.g., that the
above $\SO(3)$ action lifts to every principal $\SO(k)$ bundle with
$k\ge 5$. We will be able to answer this question almost completely.
There is only one $\Z_2$ ambiguity left as to whether the \co\
action by $\SU(2)$ on $\CP^2$ which has a fixed point lifts to an
 $\SO(k)$ principal bundle with $k\ge 5$, see
\tref{Fix}. We will see that in general, the lifting problem for
$\SO(4)$ bundles can be reduced to $\SO(3)$ bundles (see section 1).

\smallskip
We should mention that the manifold $\CP^2 \#  \CP^2$, according to
\cite{Pa} (cf. \cite{Ho}), does not admit any cohomogeneity one
action and hence the methods in this paper will not apply in this
case.

\bigskip

The paper is organized as follows. In section one we briefly recall
the basic properties of cohomogeneity one manifolds needed in our
paper, the principal bundle construction, and its characterization
in terms of the existence of commuting lifts. In section two, we
describe the cohomogeneity one actions on simply connected four
manifolds. Section three is devoted to the topological
classification of principal $\SO(k)$ bundles over 4-manifolds in
terms of invariants computable in our context. In sections four and
five  we derive which principal $\SO(k)$ bundles over any given
1-connected, 4-dimensional cohomogeneity one manifold admits a
commuting lift. Specifically, section four is devoted to the classification
over cohomogeneity one manifolds with singular orbits of codimension
two needed for the geometric consequences of the paper, and section five deals with the
classification for
cohomogeneity one actions where at most one orbit has codimension two.

\smallskip

It is our pleasure to thank V.Kapovitch,  N. Kitchloo, I. Madsen, J. Shaneson, and B.
Wilking for helpful discussions. This work was completed while the
second author was visiting IMPA and he would like to thank the
Institute for its hospitality.

\section{Basic set up}

Throughout the paper, we will make extensive use of the structure of
cohomogeneity one manifolds, which we briefly recall here for
convenience. For more details we refer to e.g. \cite{AA,Br,GZ,Mo}.

A connected manifold $M$ is said to have cohomogeneity one if it
supports a smooth action by a compact Lie group $\G$, such that the
orbit space $M/\G$ is one-dimensional. Here we are only interested
in the case where $M$ is compact and simply connected, and $\G$ is
connected. In this case $M/\G$ is an interval, and the non-principal
orbits are singular (of codimension at least two) and correspond
exactly to the end points of $M/\G$.

Fix an auxiliary $\G$-invariant metric on $M$, such that $M/\G = [-1,1]$
 isometrically, and let $c$ be a geodesic perpendicular
to all orbits. We denote by $\H$ the principal isotropy group
$\G_{c(0)}$ at $c(0)$, which is equal to the isotropy groups
$\G_{c(t)}$ for all $t\ne\pm 1$, and by $\Kpm$ the isotropy groups
at $c(\pm 1)=x_\pm$. In terms of this we have
\begin{equation}\label{discs}
M = \G \times_{\Km} \Dm \cup  \G \times_{\Kp} \Dp = M_- \cup M_+
\end{equation}

\no where  $\Dpm$ denotes the normal disc to the orbit $Gx_\pm =
\G/\Kpm = B_\pm$ at $x_\pm$, and the gluing is done along
$M_0=M_-\cap M_+=\G/\H$ with the identity map. It is important to
note that $\Sphpm =
\partial \Dpm = \Kpm/\H$, and that the diagram of groups

\begin{equation}\label{diagram}
\begin{split}
         \xymatrix{
            & {\G} & \\
            {\Km}\ar@{->}[ur]^{j_-} & & {\Kp}\ar@{->}[ul]_{j_+} \\
            & {\H}\ar@{->}[ul]^{i_-}\ar@{->}[ur]_{i_+} &
            }
\end{split}
\end{equation}

\no which we also record as $\H \subset \{\Km,\Kp\} \subset \G$,
determine $M$. Conversely, such a \emph{group diagram} with $\Kpm/\H
= \Sph^{l_{\pm}}$, defines a cohomogeneity one $\G$-manifold, given
by \eqref{discs}. The action of $G$ on $M$ is given by left
multiplication in the first component on each half and one easily
checks that this action has isotropy groups as in \eqref{diagram}.

Notice though that the description of $M$ by a group diagram
 depends on the choice of a $G$ invariant metric.  The
description of all group diagrams coming from a different choice of
a metric, or equivalently  the equivariant diffeomorphism
classification of such $\G$ manifolds, is given by:
 (cf. \cite{GWZ},
\cite{Ne},\cite{Br}):

\begin{lem}\label{equiv}
If $\H \subset \{\Km,\Kp\} \subset \G$  defines  a cohomogeneity one manifold, then
the only cohomogeneity one $\G$-manifolds which are equivariantly diffeomorphic to
it are of the form $a\H a^{-1} \subset \{ a\Km a^{-1}  , na\Kp
(na)^{-1} \} \subset\G $ for some $a\in\G$ and $n\in\N(\H)_{\subo}$.
\end{lem}

\smallskip



Next we recall from \cite{GZ}, that the above characterization of
cohomogeneity one manifolds also allows for a natural construction
of principal bundles within this category.

Let $\L$ be any compact Lie group, and $M$ a \co\ one manifold with
group diagram $\H\subset \{\Km,\Kp\}\subset \G$,  where $\G$ is
allowed to act ineffectively.

For any Lie group homomorphisms $\phi_\pm : \Kpm \to \L$ with
$\phi_+\circ i_+ = \phi_-\circ i_- = \phi_0$, let $P$ be the \co\
one $\L\times \G$-manifold with diagram

\begin{equation}\label{principal diagram}
\begin{split}
         \xymatrix{
            & {\L \times \G} & \\
            {\Km}\ar@{->}[ur]^{(\phi_-,j_-)} & &
{\Kp}\ar@{->}[ul]_{(\phi_+,j_+)} \\
            & {\H}\ar@{->}[ul]^{i_-}\ar@{->}[ur]_{i_+} &
            }
\end{split}
\end{equation}


\no The crucial property of our \co\ construction is the following
characterization:

\begin{lem}[Principal Bundle Lemma]\label{principal bundle}
Let $M$  be the \co\ $\G$-manifold defined by \eqref{diagram}.
\begin{itemize}
\item[(a)] If $P$ is a \com\ defined by
\eqref{principal diagram}, then  $\L$ acts freely on $P$, and
the quotient $P/\L$ with its induced action by $\G$ is
equivariantly diffeomorphic to $M$.
\item[(b)] Conversely, a
principal $\L$-bundle $\pi\colon P\to M$  admits a lift by $\G$,
or possibly a cover of $\G$,  that commutes with $\L$, if and
only if $P$ can be described as a cohomogeneity one $\L \times
\G$-manifold defined as in \eqref{principal diagram}.
\end{itemize}
\end{lem}

\begin{proof}
First observe that for a left action of $\L\subset \G$ on any
homogeneous space $\G/\H$, the isotropy groups are given by
$\L_{g\H}=\L\cap g\H g^{-1}$, and thus if $\L$ is normal in $\G$, it
acts freely if and only if $\L\cap \H=\{e\}$. Applying this to each
$\L\times \G$ orbit in $P$ shows that $\L$ acts freely on $P$ since
the embeddings $(\phi_\pm,j_\pm)$  are injective in the second
component. Since $\G$ is also a normal subgroup, it induces an
action on the quotient $P/\L$. Let $c(t),\ t\in [-1,1]$ be a
geodesic in an $\L\times \G$ invariant metric on $P$ normal to all
$\L\times \G$ orbits, and with isotropy groups as in
\eqref{principal diagram}. Then $c$ is a horizontal geodesic under
the projection $\pi\colon P\to P/\L$ and hence  $\pi\circ c(t)$ is a
geodesic in $M$ normal to all $\G$ orbits. Furthermore, $\pi$ is
$\G$ equivariant and since $\L$ acts transitively on the fibers of
$\pi$, it follows that $g\pi(c(-1))=\pi(c(-1))$ if and only if there
exists an $\ell\in \L$ with  $(\ell,g)c(-1)=c(-1)$ and thus
$(\ell,g)=(\phi_-(k),j_-(k))$ for some $k\in\Km$ . Similarly for
$c(1)$ and $c(0)$, which implies that \eqref{diagram} is the group
diagram for the $\G$ action on $P/\L$ and hence $P/\L$ is
$\G$-equivariantly diffeomorphic to $M$.

To prove (b), assume  there exists a cover $\sigma \colon
\tilde{\G}\to \G$ and an $\L\times\, \tilde{\G}$ action on $P$ such
that $g\cdot p=\sigma(g)\cdot\pi(p)$ for all $g\in\tilde{\G}$ and
$p\in P$. Since $P/(\L\times\, \tilde{\G})=M/\G$, the action of
$\L\times \tilde{\G}$ is \coo. We can define an ineffective action
by $\tilde{\G}$ on $M$ with isotropy groups
$\tilde{\K}^\pm=\sigma^{-1}(\Kpm)$ and $\tilde{\H}=\sigma^{-1}(\H)$
and embeddings $\tilde{j}_\pm$. For simplicity, we denote this
action again by $\G$ with $\sigma=\Id$. To determine its group
diagram, choose a metric on $P$ such that $\L\times {\G}$ acts
isometrically and such that the induced metric on $M$, which makes
$\pi$ into a Riemannian submersion, coincides with the given metric
on $M$. If we let $\tilde{c}$ be a horizontal lift of the geodesic
$c$ in $M$ normal to all $\G$ orbits, it follows that $\tilde{c}$ is
normal to all $\L\times \G$ orbits as well. Furthermore,
$(\ell,g)\tilde{c}(-1)=\tilde{c}(-1)$ implies that
$g{c}(-1)={c}(-1)$ and thus $g=j_-(k)$ for some $k\in\Km$. The
element $\ell$ is uniquely determined by $k$ since $\L$ acts freely.
Letting $\ell=\phi_-(k)$, it follows that $\phi_-$ is a
homomorphism. Hence the lift $L\times \G$ has a group diagram as in
\eqref{principal diagram}.
\end{proof}

In order to avoid having to consider covers of $\G$ as in
\lref{principal bundle} (b), we will  assume from now on that the
action of $\G$ on $M$ is an almost effective action by a simply
connected group, possibly product with a torus. This will ensure
that we obtain all possible lifts of the original action.

Notice also that in the group diagram \eqref{principal diagram}, the
action of $\L\times \G$ may not be effective, even if the action of
$\G$ is, in particular the lift may not be a product group.

\smallskip

We now collect some useful properties of these commuting lifts:

\begin{lem}\label{cover} Assume that $\pi\colon P\to M$ is an $\L$ principal
bundle and that $\G$ acts on $M$.
\begin{itemize}
\item[(a)]\label{reductionlift} If $\G$ admits a commuting lift to a reduction $P^*$ of $P$
corresponding to a subgroup $\L^* \subset \L$ it  admits a
commuting lift to $P = P^* \times_{\L^*} \L$.
\item[(b)] If $\tilde \L$ is a finite cover of $\L$ and  $P$ admits a cover $\tilde P$ yielding a
corresponding principal $\tilde \L$ \ bundle $\tilde P \to M$,
then $\tilde P$ admits a commuting lift if and only if $P$ does.
\item[(c)] \label{product} If $\L$ is a local product  $ \L_1 \cdot \L_2$ then $\G$ admits a
commuting lift to $P$ if and only if it admits a commuting lift
to  the principal $\L_1$ bundle $P_1 = P/\L_2$ and the principal
$\L_2$ bundle $P_2 = P/\L_1$.
\end{itemize}
\end{lem}

\begin{proof}
The claims in  (a) and (b)  are easily verified. In one direction
(c) follows immediately since if $\G$ admits a lift to $P$ it also
admits a lift to $P/\L'$ for any normal subgroup $\L'$ of $\L$. To
see the converse, we first reduce
 to the case of a product group $ \L=\L_1 \times \L_2$, by applying (b)
to the principal $(\L_1/F)\times (\L_2/F)$ bundle $P/F$, where
$F=\L_1\cap \L_2$ is a finite normal subgroup of $\L_i$. In the case
of a product group, the $ \L=\L_1 \times \L_2$ principal bundle $P$
has classifying map  given by  $f =(f_1,f_2)\colon M\to B_{\L_1}
\times B_{\L_2}$ where $f_i$ are the classifying maps of $P_i$.
Hence the $\L$ principal bundle $P$ is determined up to isomorphism
by the $\L_i$ principal bundles $P_i$. Now consider
$\tilde{P}=\{(x_1,x_2)\in P_1\times P_2 \mid
\pi_1(x_1)=\pi_2(x_2)\}$. $\tilde{P}$ is clearly  a principal
$\L_1\times \L_2$ bundle over $M$ and since $\tilde{P}/\L_i=P_i$ the
bundle $\tilde{P}$ must be isomorphic to $P$. If $P_i$ now admits a
lift that commutes with $\L_i$, it clearly also admits a lift to
$\tilde{P}$ that commutes with $\L_1 \times \L_2$.
\end{proof}

In the case of \co\ one actions we have:

\begin{lem}\label{reduction}
Let $(M,\G)$ be a cohomogeneity one manifold   as in \eqref{diagram}
and $(P,\L\times \G)$ an $\L$ principal bundle over $M$ as in
\eqref{principal diagram}. Then we have:
\begin{itemize}
\item[(a)] If a subgroup $\L^* \subset \L$ contains the image groups $\phi_\pm
(\Kpm)$ then $P$ admits a reduction to $\L^*$.
\item[(b)] Suppose $\G$ is a local product $\G_1\cdot\G_2$ and the
subaction of $\G_1\times \{e\}\subset\G_1\cdot\G_2$ has the same
orbits as the $\G$ action. Then the action of $\G_1\cdot\G_2$
admits a commuting lift to $P$ if and only if $\G_1$ does.
\end{itemize}
\end{lem}

\begin{proof} To see (a), consider the \com\ $P^*$ defined by $\H\subset \{
\Km,\Kp\}\subset \L^*\times \G$.  Then $P=P^*\times_{\L^*}\L$ which
one easily verifies by showing that the $\L\times \G$ actions  on
both have the same isotropy groups . Hence $P$ reduces to $P^*$.

In order to prove (b) we assume, by making the action almost
effective if necessary, that $\G_1$ and $\G_2$ are simply connected
and that $\G=\G_1 \times \G_2$. Suppose the $\G_1$ action on $M$
with diagram $\H \subset \{\Km,\Kp\} \subset \G_1$ lifts. We first
claim that the diagram for the extended almost effective $\G_1
\times \G_2$ action on $M$ has isotropy groups $\H \times \G_2
\subset \{\Km \times \G_2,\Kp \times \G_2\}\subset \G_1 \times
\G_2$, where the $\G_2$ factor is embedded diagonally in $\G_1
\times \G_2$. Indeed, if $\G_1\times \G_2/\S$ is a homogeneous space
where $\G_1\times \{e\}\subset \G_1 \times \G_2$ acts transitively,
then $\S$ projects onto the second factor in  $\G_1\times \G_2$ and
$\G_1\times \G_2/\S=\G_1/(\G_1\times\{e\})\cap \S$. Thus
$\S=\S_1\cdot\S_2$ with $\S_1=(\G_1\times\{e\})\cap \S$ and $\S_2$ a
complementary normal subgroup. Since $\S_2$ projects onto $\G_2$ and
$\G_2$ is simply connected, it follows that $\S_2\simeq\G_2$.
Furthermore, $\S_1\cap\S_2=\{e\}$ and thus $\S$ is the direct
product $\S=\S_1\times\S_2$ with $\S_2$  embedded diagonally. We now
apply this argument to each orbit.

 The desired homomorphisms in the
construction of the lift of the $\G_1 \times \G_2$ action can thus
be taken to be the projection to the first factor followed by the
homomorphisms used in the construction of the $\G_1$ lift.
\end{proof}

Lemma \ref{product} (c)  will be particularly useful for us in the
case of $\L=\SO(4)$. In \SO(4) there are two normal subgroups
$\S^3_-, \S^3_+$ isomorphic to $\S^3$, defined by left and right
multiplication by unit quaternions on $\QH\simeq \R^4$ and
$\SO(4)/\S^3_\pm$ is isomorphic to $\SO(3)$. Hence, if $\SO(4)\to P
\to M$ is a principal \SO(4) bundle, there are two associated
principal \SO(3) bundles $P^\pm=P/\S^3_\mp$ over $M$. The
relationship between these bundles can be described as follows:

\begin{prop}\label{so4}
Let $M$ be a compact, simply connected n-dimensional manifold and
$\pi\colon P\to M$ a principal $\SO(4)$ bundle with associated
principal $\SO(3)$ bundles $\pi_\pm\colon P^\pm \to M$.  The
principal bundle  $P$ is uniquely determined by the $\SO(3)$ bundles
$P^\pm$ and  $w_2(P)=w_2(P^\pm)$. Conversely, if $P^\pm\to M$ are
two principal $\SO(3)$ bundles with $w_2(P^-)=w_2(P^+)$, then there
exists a principal $\SO(4)$ bundle $P\to M$ which gives rise to
$P^\pm$.
\end{prop}

\begin{proof}
To see that $w_2(P)=w_2(P^\pm)$, consider  the following
commutative diagram of homotopy sequences:
\begin{equation} \label{w2}
\begin{split}
\xymatrix{ 0 \ar[r] & \pi_2(P) \ar[r] \ar[d] & \pi_2(M)
\ar[d]\ar[r]^{w_2(P)} &  \pi_1(\SO(4))\ar[d]\ar[r]&
\pi_1(P)\ar[d]\ar[r] & 0 \\
0 \ar[r] & \pi_2(P^{\pm}) \ar[r]  & \pi_2(M) \ar[r]^{w_2(P^\pm)} &
\pi_1(\SO(3))\ar[r]& \pi_1(P^\pm)\ar[r] & 0
     }
\end{split}
\end{equation}

\no where the boundary map $\partial:\pi_2(M)\to
\pi_1(\SO(k))=\Z_2$, if considered as an element of $\Hom
(\pi_2(M),\Z_2)=H^2(M,\Z_2)$, is precisely $w_2$ (This follows
e.g., by observing that it is clearly true for the universal
bundle). 
Since $\pi_1(\SO(4))\to \pi_1(\SO(3))$ is an
isomorphism, it follows that $w_2(P)=w_2(P^\pm)$. Furthermore, $P$
and $ P^\pm$ are simply connected if and only if $w_2\ne 0$ and
their fundamental group is $\Z_2$ if $w_2=0$.

Now assume that  $w_2(P^\pm)=0$. The bundles $P^\pm$ then admit
(unique) two fold covers $\tilde{P}^\pm$ which are principal $\S^3$
bundles over $M$ classified by $f_\pm\colon M\to B_{\S^3}$ and $f
=(f_-,f_+)\colon M\to B_{\S^3\times \S^3}$ defines a principal
$\S^3\times \S^3$ bundle $\tilde{P}$ over $M$. The bundle
$P=\tilde{P}/\{\pm(1,1)\}$ is now the desired $\SO(4)$ principal
bundle and uniqueness follows as well since $\pi_1(P)=\Z_2$.

 If $w_2(P^-)=w_2(P^+)\ne 0$, we consider, as in the proof of \lref{cover},
  $\tilde{P}=\{(x_-,x_+)\in P^-\times P^+
\mid \pi_-(x_-)=\pi_+(x_+)\}$ together with the natural principal
\SO(3) bundle projections $\gs_\pm\colon \tilde{P}\to P^\pm$.
$\tilde{P}$ is clearly also a principal $\SO(3)\times \SO(3)$ bundle
over $M$. $\tilde{P}$ can be regarded as the pullback of the
principal \SO(3) bundle $P^+ \to M$ via $\pi_-$, and also as the
pullback of the principal \SO(3) bundle $P^-\to M$ via $\pi_+$,
i.e., we have the following commutative diagram of pullback bundles:

\begin{equation*}
\begin{split}
\xymatrix{\tilde{P} \ar[d]_{\gs_-}\ar[r]^{\gs_+}& P^+  \ar[d]^{\pi_+}  \\
P^- \ar[r]^{\pi_-} & M }
\end{split}
\end{equation*}

For both ways of looking at $\tilde{P}$, it follows that
$w_2(\tilde{P})=0$ since the compositions $ \pi_2(P^\pm)  \to
\pi_2(M) \overset{{w_2(P)}}{\longrightarrow} \pi_1(\SO(3))$
     are 0, which implies that
     $0=\pi_-^*(w_2(P^-))=\pi_-^*(w_2(P^+))=
     w_2(\tilde{P})$ and
similarly for $\pi_+$.  Furthermore, since $w_2(P^\pm)\ne 0$,
$P^\pm$ are simply connected, and since $w_2(\tilde{P})=0$ we have
$\pi_1(\tilde{P})=\Z_2$. Hence the unique two fold cover $P\to
\tilde{P}$ is a spin cover of each bundle $\gs_\pm\colon
\tilde{P}\to P^{\pm}$. But \SO(4) is the only two fold cover of
$\SO(3)\times\SO(3)$ which is a two fold cover along each \SO(3) and
hence $P$ is a principal \SO(4) bundle which clearly gives rise to
$P^{\pm}$ in the original construction. Uniqueness follows from the
same commutative diagram of principal bundles.
\end{proof}

We furthermore remark  that the principal $\SO(4)$ bundles whose
structure group reduces to $\U(2)$, i.e. the complex vector bundles,
(resp. $\SU(2)$), are precisely those where either $P^+$ or $P^-$,
reduces to an $\SO(2)$ bundle (resp. becomes trivial).  The even
more special $\SO(4)$ bundles where the structure group reduces to a
2-torus, i.e. the direct sum of two complex line bundles, correspond
to those where both $P^+$ and $P^-$ reduce to an $\SO(2)$ bundle.
Finally, the bundles where the structure group reduces to $\SO(3)$
(i.e. the bundles with a section) correspond to the ones where $P^+$
and $P^-$ are isomorphic.

In terms of oriented vector bundles, the above relationship between
$P$ and $P^\pm$ can also be described as follows. If E is the 4
dimensional vector bundle over M with principal bundle $P$, then
$\Lambda^2(E)=\Lambda^2_-(E)\oplus \Lambda^2_+(E)$ is given by the
decomposition of a 2 form into its self dual and anti self dual
part. Then $\Lambda^2_\pm(E)$ is the 3 dimensional vector bundle
whose principal $\SO(3)$ bundle is $P^\pm$, which follows from the
fact that the decomposition
$\fso(4)=\Lambda^2\R^4=\Lambda^2_-\R^4\oplus
\Lambda^2_+\R^4=\fso(3)\oplus\fso(3)$ is the decomposition of
$\fso(4)$ into simple ideals. Using this, one easily shows (cf.
\cite{DR}):

\begin{equation}\label{ppm}
p_1(P^\pm)=p_1(P)\pm 2e(P)
\end{equation}

In the case of  complex rank two bundles one has the Chern classes
$c_1$ and $c_2$ and since for the underlying real bundle
$p_1(P)=c_1^2-2c_2$ (cf. \cite[15.5]{Mi}), the relationship
\eqref{ppm} implies that:

\begin{equation}\label{chern}
p_1(P^+)=c_1^2 \text{ and } p_1(P^-)=c_1^2-4c_2
 \end{equation}
and hence, under the usual embedding of $\U(2)$ in $\SO(4)$, $P^+$
is the one whose structure group reduces.

\section{Cohomogeneity one four manifolds}

\smallskip

According to Parker's classification \cite{Pa} of all cohomogeneity
one 4-manifolds, the only simply connected manifolds which admit
such actions are $ \Sph^4, \CP^2, \Sph^2 \times \Sph^2, \text{and }
\CP^2 \# -\CP^2$. We will analyze the lifting problem for each
cohomogeneity one action by a connected compact group on these
manifolds.

\smallskip

In this section we describe the \co\ one actions on simply connected
4-manifolds. Although these actions are exhibited in \cite{Pa}, we
need to know the precise group picture for our applications. Recall
that we can derive the group diagram \eqref{diagram} by choosing one
singular orbit $B_- = \G/\Km$ with $\Km = \G_{x_-}$  and by choosing
any geodesic $c(t)=\exp_{x_-}(tv_-)$ with $v_-$ normal to $B_-$
since it will then  automatically be normal to all orbits. We then
need to determine the first $t_0$ when $c(t_0)$ meets $B_+$, i.e.
when $\G_{c(t_0)}$ is bigger than the principal isotropy group $\H$.
Then $c(t_0)=x_+$ and $\Kp=\G_{x_+}$ and $\H=\G_{c(t)} , 0<t<t_0$.

\smallskip

We start with cohomogeneity one actions with singular orbits of
codimension two, since they are the most important ones in geometric
applications.  
In addition, with one exception, they are all \coo\  under
$\G=\SO(3)$ or $\SU(2)$ or an extension of an $\SU(2)$ cohomogeneity one action
to $\U(2)$. We will describe them as an action with $\G=\S^3$ in
order to obtain all possible lifts.

\smallskip

First recall that the linear $\SO(3)$ action on $\Sph^4$
corresponding to a maximal subgroup of $\SO(5)$, which played a
pivotal role in \cite{GZ}, has the following group diagrams when
lifted to $\S^3$:
\begin{equation}\label{S^4}
\begin{split}
         \xymatrix{
            & {\S^3} & \\
           { {\text{C}_i}\cup j{\text{C}_i}}\ar@{->}[ur]^{}
           & & {{\text{C}_j}\cup i
{\text{C}_j}}\ar@{->}[ul]_{}
\\
            & {\text{Q}}\ar@{->}[ul]^{}\ar@{->}[ur]_{} &
            }
\end{split}
\qquad \Sph^4
\end{equation}

\no where C$_i =\{ {e^{i\gt}} | \gt \in \R\}$, C$_j =\{ {e^{j\gt}} |
\gt \in \R\}$ are ``coordinate'' circle groups, and ${\text{Q}} = \{\pm 1, \pm i, \pm j, \pm k\}$ is the quaternion group.

Next consider $\CP^2$, where we write a point in homogeneous
coordinates $[v] , v=(z_0,z_1,z_2)$ $\in \C^3$. Then $\SO(3)$ acts
on $\CP^2$ as $[v] \to [gv]$. One singular orbit is clearly
$B_-=\RP^2=\{[v]\mid v\in \R^3\subset \C^3\}$. Let $x_-=[(1,0,0)]$
and hence $\Km=\orth{1}{2}$. One easily checks that
$c(t)=[(\cos(t),i\sin(t),0)]$ is a geodesic in $\CP^2$ orthogonal to
$B_-$ at $x_-$ and that $\H=\G_{c(t)}=\Z_2=
\langle\diag(-1,-1,1)\rangle$ as long as $0<t<\pi/4$. When $t=\pi/4$
we set $x_+=[(1/\sqrt{2},i/\sqrt{2},0)]$ and one shows that
$\Kp=\G_{x_+}=\SO(2)$ which is embedded in $\SO(3)$ as a rotation in
the first two coordinates. $B_+=\G/\Kp$ can also be described as the
quadric $\sum z_i^2=0 $ in $\CP^2$ since $x_+$ lies in it and the
quadric is clearly preserved by the $\SO(3)$ action.  After lifting
these groups into $\S^3$ (and renumbering the coordinates)  the
group picture for $\CP^2$ becomes:

\begin{equation}\label{CP2}
\begin{split}
         \xymatrix{
            & {\S^3} & \\
            {\text{C}_i \cup j\text{C}_i}\ar@{->}[ur]^{} & &
{\text{C}_j}\ar@{->}[ul]_{} \\
            & {\Z_4= \langle j \rangle }\ar@{->}[ul]^{}\ar@{->}[ur]_{} &
            }
\end{split}
\qquad \CP^2
\end{equation}

If we compare these group diagrams, one immediately sees that
$\CP^2$ is an equivariant two  fold branched cover of $\Sph^4$, with
branching locus (and metric singularity) along  the real points
$\RP^2$ in $\CP^2$ and   the Veronese embedding of $\RP^2$ in
$\Sph^4$, the covering given by coverings along the orbits of
$\S^3$.

Next we consider the \com \ $M_n$ defined by the group diagram:

\begin{equation}\label{Mn}
\begin{split}
         \xymatrix{
            & {\S^3} & \\
            {\text{C}_i}\ar@{->}[ur]^{} & & {\text{C}_i}\ar@{->}[ul]_{} \\
            & \Z_n \ar@{->}[ul]^{}\ar@{->}[ur]_{} &
            }
\end{split}
\end{equation}

One easily shows that for $n=1$ this is the action on $\C
P^2\#-\CP^2$ obtained as follows: $\SU(2)$ acts on $\CP^2$ fixing a
point. Take two copies of $\CP^2$ with such an action and remove a
small ball around the fixed points. If we identify the boundaries
with an equivariant diffeomorphism, we obtain the desired action on
$\C P^2\#-\CP^2$. For $n=2$, one easily shows that the group diagram
is the one induced by the linear action of $\SO(3)$ on
$\Sph^2\times\Sph^2$ given by  $A\cdot (v,w)=(Av,Aw)$. In both cases
we have used the fact (see Lemma \ref{equiv}) that  any two circles
in $S^3$ can be conjugated to each other with an element in
$\N(\H)_{\subo}=S^3$. In general one shows ( \cite{Pa} ) that $M_n$
is diffeomorphic to $\CP^2\#-\CP^2$ for $n$ odd, and to
$\Sph^2\times\Sph^2$ for $n$ even, although they are of course not
equivariantly diffeomorphic. All actions in \eqref{Mn} admit
extensions to $\S^3\times \S^1$ given by the group diagrams
$\S^1=\{(e^{in\theta}, e^{i\theta})\} \subset \{\T^2,\T^2\} \subset
\S^3\times \S^1$.

\smallskip

The only further action where both orbits have cohomogeneity two is
the product of a transitive action and a \co\ one action on $\Sph^2
\times \Sph^2$, i.e., it is given by

\begin{equation}\label{trans}
 \SO(2) \times
\{1\} \subset \{\SO(2) \SO(2), \SO(2) \SO(2)\} \subset \SO(3) \SO(2)
\end{equation}

\bigskip

We now proceed to quickly record the  remaining cohomogeneity one
actions on 1-connected 4-manifolds. 

\smallskip

 The first one is the suspension action on $ \Sph^4 \subset
  \R^5=\R^4\oplus\R$ with diagram:

\begin{equation}\label{sum1}
 \{1\} \subset \{\S^3,\S^3\} \subset \S^3
\end{equation}

\no and its extensions

$$ \U(1) \subset \{ \U(2),\U(2)\} \subset \U(2)  \text{ and
}  \SO(3) \subset
\{\SO(4),\SO(4)\}  \subset \SO(4)$$

\smallskip

\no The second is a sum action on $\Sph^4 \subset
\R^5=\R^3\oplus\R^2$ with diagram:

\begin{equation}\label{sum2}
\{1\} \times \SO(2) \subset \{\{1\} \times
\SO(3),\SO(2) \times \SO(2) \} \subset \SO(2) \SO(3)
\end{equation}

\smallskip

\no The third action is the action on $\CP^2$ with a fixed point,
i.e., it is induced
 from the sum action on $\Sph^5 \subset \C \oplus \C^2$ and has the
  diagram

\begin{equation}\label{cp2fix}
 \{1\} \subset
\{\SU(2), \U(1)\} \subset \SU(2)
\end{equation}

\no and its extension
$$ \U(1) \subset \{\U(2), \U(1) \U(1)\} \subset
\U(2)$$

\smallskip

\no According to Parker \cite{Pa}, this exhaust all cohomogeneity
one actions on 1-connected 4-manifolds. But notice that the action
of $\SO(3)$ on $\CP^2$ and all actions with $\G=\U(2)$ were left out
in his classification. For a complete list see \cite{Ho},  where a
classification of simply connected \com s of dimension at most $7$
was carried out.

\section{Topological Classification}

The purpose of this section is to review the classification of
vector bundles over simply connected closed 4-manifolds $M$, and to
relate it to our setting. Specifically, the classification is expressed in terms of characteristic classes, and for our purposes it is essential that we can read this information off from the cohomology of the total space of the corresponding principal bundles. Since all vector bundles over $M$ are
orientable, this amounts to a classification of principal $\SO(k)$
bundles over such manifolds.

\begin{equation}
Principle \  \SO(2) \  bundles \  P \ are \ classified \ by
\ their \ Euler class
 \ e(P)\in H^2(M,\Z)
 \end{equation}

\no and any such class is realized by a unique principal bundle.

\smallskip

In the case of principal $\SO(3)$ bundles, it is well known that
they are classified  by their second Stiefel Whitney class $w_2(P)
\in H^2(M,\Z_2)$  and their first Pontryagin class $p_1(P) \in
H^4(M,\Z) = \Z$ (cf. \cite{FU} and \cite{DW}).  Fixing an
orientation on $M$, we identify $p_1$ with the integer
$k=p_1(P)([M])$.
    Here  $w_2$ can be chosen arbitrary, but the
 values of $p_1$ are restricted to a congruence class mod $4$.
     To see which one is allowed for a given value of $w_2(P)$,
choose a principal \SO(2) bundle  $P^*$ over $M$ whose Euler class
$e\in H^2(M,\Z)$ reduced mod 2 is equal to $w_2$ and let $P$ be the
     principal \SO(3) bundle that we obtain by extending the
     structure group of
$P^*$.
     Then  $w_2(P)=w_2$ and $p_1(P)=e^2$
(see \cite[15.8]{Mi} ). Thus we have:

\begin{prop}\label{wep}
A principal $\SO(3)$ bundle $P$ over a simply connected 4-manifold
$M$  is determined by its second Stiefel Whitney class $w_2$ and its
first Pontryagin class $p_1$. Here $w_2$ can take any value, and
$p_1$ can take any value congruent to $e^2$ mod 4, where $e$ is the
Euler class of a principal circle bundle with Stiefel Whitney class
$w_2$.
\end{prop}

In particular, for bundles over $\CP^2$, we have  $\con p_1=1(4) $
if $w_2\ne 0$ and $\con p_1=0(4) $ if $w_2 = 0$, as long as we
choose the orientation class on $\CP^2$ to be the square of a
generator in dimension 2.

In Section 1 we saw how principal $\SO(4)$ bundles $P$ are in
 one to one correspondence with a pair of $\SO(3)$ bundles
 $P^{\pm}$ with the same $w_2$. They are classified according
 to the following well know result (cf. \cite{DW}) :

\begin{prop}\label{prinso4}
A principal $\SO(4)$ bundle $P$ over a simply connected 4-manifold
$M$ is determined by $w_2(P)$, $p_1(P)$, and the Euler class $e(P)$.
Here $w_2(P)$ can take any value, whereas  $p_1(P)$ and $e(P)$ are
restricted via \pref{wep} and \eqref{ppm}.
\end{prop}

In particular for $\SO(4)$ bundles over $\CP^2$,   the
allowed values of these invariants are  given by
$p_1(P)=2k+2l\, ,\, e(P)=k-l$ for $ k,l\in\Z$ in the spin case, and
by $p_1(P)=2k+2l+1 \, ,\, e(P)=k-l$ in the non-spin case.

\bigskip

The classification of the remaining principal bundles is provided by  the
following well known fact (cf. \cite{DW}):

\begin{prop}\label{prinsok}
For $k \ge 5$, a principal $\SO(k)$ bundle $P$ over a simply
connected 4-manifold $M$ is determined by $w_2(P)$, $w_4(P)$ and
$p_1(P)$.
\end{prop}

Recall  that a $k$ dimensional vector bundle over $M^4$ is the
direct sum of a $4$ dimensional vector bundle with a trivial one and
that for a $4$ dimensional vector bundle  $w_4(P)\equiv e(P)$ mod
$2$. This completely determines the allowed values of the invariants
in \pref{prinsok}.

\smallskip

In our later applications we only need to consider $M=\Sph^4$ and
$M=\CP^2$ and in these cases one obtains the following table for the
allowed values:

\renewcommand{\thetable}{\Alph{table}}
\stepcounter{equation}

\setlength{\tabcolsep}{0.40cm}
\renewcommand{\arraystretch}{1.6}
\begin{table}[!h]
\begin{center}
\begin{tabular}{|c|c|c|}
  \hline
  $p_1$ & $w_2$ & $w_4$ \\ \hline
  0 mod 4 & 0 & 0 \\ \hline
  2 mod 4 & 0 & $\ne 0$ \\ \hline
  1 mod 4 & $\ne 0$ & 0 \\ \hline
  3 mod 4 & $\ne 0$ & $\ne 0$ \\
  \hline
\end{tabular}
\end{center}\caption{}
\end{table}
\bigskip

 In particular, we see that the values of $w_2$ and $w_4$ happen to be
determined by $p_1$.

\smallskip

Our strategy in determining which principal $\SO(k)$ bundles $P$
admit a commuting lift of a given cohomogeneity one action is to use
their description in Lemma \ref{principal bundle} and compute the
possible corresponding characteristic classes. As indicated earlier, it is crucial for us
that these in turn can be expressed in terms of the topology of the
total space $P$ according to the following result:

\begin{prop}\label{HP}
Let $M$ be a compact, simply connected, 4-dimensional manifold with
second Betti number $b$ , and $P\to M$ a principal $\SO(k)$ bundle
with $s=|p_1(P)([M])|$. Then
\begin{itemize}
\item[(a)] $w_2\ne 0$ if and only if $P$ is simply connected.
\item[(b)]$p_1(P) \ne 0$ if and only if $H^4(P,\Z)$ is finite,
 and in that case
$H^3(P,\Z)=0$.
\item[(c)] If   $p_1(P)\ne 0$, then
$|H^4(P,\Z)| = 2^{b - 1}s$ if $k=3$ and $|H^4(P,\Z)| = 2^{b}s$ if
$k\ge 5$.
\end{itemize}
\end{prop}
\begin{proof}
As we observed in the proof of \pref{so4}, it follows from the long
homotopy sequence of the principal bundle that $\pi_1(P)=0$
precisely when $w_2\ne 0$.

To compute $H^4(P)$ we use the spectral sequence for the principal
bundle $\SO(k)\to P\to M$.  Let $(E_r,d_r)$ be the spectral sequence
of this bundle, and $(\tilde{E}_r,\tilde{d}_r) $ the spectral
sequence of the universal principal $\SO(k)$ bundle $\SO(k)\to E\to
B_{\SO(k)}$ with contractible total space $E$. The classifying map
$f\colon M\to B_{\SO(k)}$ induces maps between these spectral
sequences and we will use the naturality of the differentials.

We first assume that $k\ge 5$ and examine the spectral sequence for
the universal bundle. It is well known that for the  cohomology
groups $H^*(\SO(k),\Z)$  one has: $H^1=0 , H^2=\Z_2 , H^3= \Z ,
H^4=\Z_2$ (see e.g., \cite{Ha} p. 292). Similarly, the groups
$H^*(B_{\SO(k)},\Z)$ are given by $H^1=H^2=0 , H^3=H^5=\Z_2$.
Moreover, $H^4(B_{\SO(k)})=\Z$ with generator $p_1(E)$ (see e.g.
\cite[p.182]{Mi}).
  Hence
$\tilde{E}_2^{2,2}=H^2(B_{\SO(3)},H^2(\SO(3)))=H^2(B_{\SO(3)},\Z_2)=\Z_2$.
and since $E$ is contractible, $\tilde{d}_2\colon \tilde{E}_2^{0,3}
= H^3(\SO(3),\Z)=\Z \to \tilde{E}_2^{2,2}=\Z_2$ must be onto and
hence $\tilde{d}_2x\ne 0$ if $x\in \tilde{E}_2^{0,3}$ is a
generator. If we reduce the coefficients from $\Z$ to $\Z_2$ in the
spectral sequence, this element $\tilde{d}_2x$ is non-zero, and
hence corresponds to the second Stiefel Whitney class. Thus we can
write $\tilde{d}_2x = w_2(E)$. Furthermore $\tilde{d}_4\colon
\tilde{E}_4^{0,3}=\Z\to \tilde{E}_4^{4,0}=\Z$ must take the
generator $2x$ in $\tilde{E}_4^{0,3}$ to the generator $p_1(E)$ in
$\tilde{E}_4^{4,0}$, i.e. $d_4(2x)=p_1(E)$.

Notice that under the map $f$, $f^*(p_1(E)) = p_1(P)\in H^4(M,\Z)$
and $f^*(w_2(E)) = w_2(P)\in H^2(M,\Z_2)$ and hence the naturality
of the spectral sequence implies that $d_2x =w_2(P)$ and $d_4 (2x)
=p_1(P)$. Therefore, if $w_2(P)\ne 0$, we have $d_2x\ne 0$ and $2x$
becomes the generator in $E_3^{0,3}=\Z$ with $d_4(2x)=p_1(P)$. Thus
$E_5^{4,0}=E_{\infty}^{4,0} = \Z_s $ and $E_5^{2,2}
=E_{\infty}^{2,2} = \Z_2^{b-1}$, $E_5^{0,4} =E_{\infty}^{0,4} =
\Z_2$ and hence $\left| H^4(P,\Z)\right| = 2^{b}s$ if $p_1(P) \ne
0$. If, on the other hand $p_1(P) = 0$, $E_{\infty}^{4,0}= \Z$ which
survives in $H^4(P,\Z)$.

If $w_2(P)=0$, then $d_2x=0$ and hence $x$ is a generator in
$E_3^{0,3}=\Z$  with $d_4x=\frac{1}{2}p_1(P)$. Thus
$E_{\infty}^{4,0} =\Z_{\frac{1}{2}s} $  and $E_{\infty}^{2,2} =
\Z_2^{b}$  if $p_1(P)\ne0$, and we obtain again $\left|
H^4(P,\Z)\right| = 2^{b}k$. As before, if $p_1(P)=0$, $H^4(P,\Z)$
contains $\Z$.

In both cases, the fact that $H^3(P,\Z)=0$  if $p_1(P)\ne0$ now
immediately follows from the spectral sequence for $P$.

The only change that occurs when $k=3$ is that
$E_\infty^{0,4}=H^4(\SO(k),\Z)=0$ and $H^5(B_{\SO(k)})=0$. Arguing
as above finishes the proof.
\end{proof}

\smallskip

This leaves undetermined the sign of $p_1$.   Notice that in the
case of $k=3$ and $w_2= 0$,
 the two fold spin cover of $P$ is an $\S^3$
bundle over $M$ whose Pontryagin class is four times the Euler class
since $H^4(B_{\SO(3)})\to H^4(B_{\SU(2)} )$ is multiplication by
$4$. Hence one can change the sign by changing the orientation of
the sphere bundle.  Thus, if one admits a lift, so does the other.
In the case of $k=3$ and $w_2\ne 0$ on the other hand, the sign of
$p_1(P)$ is determined since $p_1(P)\equiv e^2$ mod $4 $.

For $k>4$ the situation is more complicated. One can show that the
two bundles obtained by changing the sign of $p_1$ have the same
cohomology groups. As was pointed out to us by N.Kitchloo, the
homotopy type of the two bundles is different though and can be
distinguished by the Steenrod square $Sq^2$. This issue will arise
for us only in the case of the $\SU(2)$ action on $\CP^2$ with a
fixed point, where we will leave it as an open question. Notice
though that this does not effect the geometric applications in the
Introduction since for this action only one of the singular orbits
has
 codimension two.

\smallskip

The topology of principal bundles $P$ that admit a commuting lift
can, in our case, be analyzed in terms of their description
\eqref{principal diagram} as cohomogeneity one manifolds. In
particular, their decomposition as a union of two disc bundles over
the singular orbits allow the use of known
 topological tools and representation theory to complete our task in the
  next two sections.

\section{Lifts yielding  bundles with nonnegative curvature}

The following result obtained  in \cite{GZ} is the basic source
 for nonnegative curvature examples we use here.

\begin{thm}\label{curvature}
Any cohomogeneity one manifold with codimension two singular orbits
admits a nonnegatively curved invariant metric.
\end{thm}

A key property of the principal bundle construction $P \to M$ for
cohomogeneity one manifolds (see \ref{principal bundle}) is that
the normal bundles to the singular orbits in $P$ are the pull backs
of the normal bundles of the singular orbits in the base $M$. In
 particular, $P$ admits an invariant metric of nonnegative curvature
  if the singular orbits of the base have codimension two.

\bigskip

We will now begin our classification of principal $\SO(k)$ bundles
admitting commuting lifts. This section is devoted to lifts of
actions where both singular orbits have codimension two.

\smallskip

 Recall (\cite{PS},\cite{HY}) that any
group acting on a simply connected manifold $M$ admits a unique
commuting lift to any torus bundle over $M$. We can therefore assume
that $k\ge 3$.

We first deal with the simpler case of bundles over
$\Sph^2\times\Sph^2 \ \text{and } \CP^2 \# - \CP^2$ given by the
cohomogeneity one actions \eqref{Mn}, their extensions to
$\S^3\times \S^1$,  or \eqref{trans}. Since for these actions
$\Km=\Kp$ is a maximal torus of $G$,
 \lref{reduction} (a) implies that  the structure group reduces to a torus. Thus we have:

\begin{thm}\label{toruslift}
A principal $\SO(k)$-bundle over  $ \Sph^2\times \Sph^2$ or $\CP^2
\# - \CP^2$ with any of its cohomogeneity one actions admits a
commuting lift if
and only if it admits a reduction to a torus.
  \end{thm}

It is worth noticing that over $ \Sph^2\times\Sph^2 \text{ and }
\CP^2 \# - \CP^2$ quite a few of the principal $\SO(k)$ bundles
already arise in this trivial fashion. More precisely, we have:

\begin{thm}\label{product-sum}
\begin{itemize}
\item[(a)] A principal $\SO(3)$-bundle over  $ \Sph^2\times \Sph^2$ admits
a reduction to $\SO(2)$ unless $w_2 = 0$ and $p_1 \equiv 4$ mod $8$
( $0$ mod $4$ are the allowed values).
\item[(b)]
A principal $\SO(3)$-bundle over   $\CP^2 \# - \CP^2$ admits a
reduction to $\SO(2)$ unless $w_2 = 0$ and $p_1 \equiv 8$ mod $16$,
or $w_2 = (1,1)$ and $p_1 \equiv 4$ mod $8$ (in both cases the
allowed values for $p_1$ are $0$ mod $4$).
\item[(c)] Every principal $\SO(k)$ bundle over $ \Sph^2\times\Sph^2 \text{ and }
\CP^2 \# - \CP^2$ with $k\ge 6$ admits a reduction to a $3$-torus.
\end{itemize}
\end{thm}

\begin{proof}
 If $M=\Sph^2\times \Sph^2$,   the $\SO(2)$ principal bundle over
$M$ with Euler class $e=(a,b)\in H^2(M,\Z)=\Z\oplus\Z$ has first
Pontryagin class $p_1=(a,b)^2[M]=2ab$. We now use \pref{wep} to
determine the allowed values of $p_1$ for an $\SO(3)$ principal
bundle with a given Stiefel Whitney class $w_2=(a,b) \mod 2$. The
circle bundle with $e=(2k,1)$ (resp. $e=(1,2k)$) has $p_1=4k$ and
$w_2=(0,1)$ (resp. $(1,0)$) which are precisely the allowed values
in this case. Thus all $\SO(3)$ bundles with $w_2=(0,1)$ and $(1,0)$
admit a reduction to $\SO(2)$. The same holds in the case of
$w_2=(1,1)$ using the 2-plane bundle with $e=(2k+1,1)$. In all 3
cases \pref{so4} implies that a bundle with $k=4$ admits a reduction
to a 2-torus since both $P^\pm$ admit a reduction to $\SO(2)$. For
$k\ge 5$  the  structure group reduces to $\SO(4)$ and hence to a
2-torus.

If $w_2=(0,0)$ the 2-plane bundle with $e=(2k,2l)$ has $p_1=8kl$.
Hence the $\SO(3)$ bundles with $p_1 \equiv 0\mod 8$ reduce to
$\SO(2)$ and those with $p_1 \equiv 4\mod 8$ do not. If we let $P^-$
be the $\SO(3)$ bundle with $p_1=8k$ and $P^+$  the one with
$p_1=8l$, then \pref{wep} defines an $\SO(4)$ bundle with
$w_2=(0,0)\ ,p_1=4k+4l$ and $e=2k-2l$. Thus, according to
\tref{prinsok}, any bundle with $k\ge 5 $, $w_2=(0,0)$, $p_1\equiv 0
\mod 4$ and $w_4=0$ admits a reduction to a torus. To see what the
precise allowed values are in this case we use \pref{prinsok}. Since
every $k$ dimensional bundles reduces to $\SO(4)$ and since the
$\SO(4)$ bundle with $p_1(P^-)=4k$ and $p_1(P^+)=4l$ has $p_1=2k+2l$
and $e=k-l$, it follows that for $k\ge 5$ we either have $p_1\equiv
0 \mod 4$ and $w_4=0$ or $p_1\equiv 2 \mod 4$ and $w_4\ne 0$. In the
latter case we can consider the direct sum of $3$ two dimensional
bundles with Euler classes $e=(1,1)$, $e=(1,0)$ and $e=(2k,1)$ which
has $w_2=(0,0)$ and $p_1=4k+2$ (and thus $w_4\ne 0$). Thus its
structure group reduces to a 3-torus.

We indicate the argument for $\CP^2 \# - \CP^2$. The $\SO(2)$ bundle
with Euler class $e=(a,b)\in H^2(M,\Z)=\Z\oplus\Z$ has first
Pontryagin class $p_1=(a,b)^2[M]=a^2-b^2$. If $w_2=(1,0)$ or $(0,1)$
the bundles with $e=(2k+1,2k)$ resp. $e=(2k,2k-1)$ have $p_1=4k+1$
resp. $p_1=4k-1$ and these are precisely the allowed values in the
case of $k=3$. Thus any $\SO(k)$ bundle with these Stiefel Whitney
classes reduces to a 2-torus.

If  $w_2=(1,1)$ the 2-plane bundle with $e=(2k+1,2k-1)$ has
$p_1=8k$, whereas $p_1\equiv 0\mod 4$ are the allowed values in the
case of $k=3$. This gives rise to 5-plane bundles with $w_2=(1,1)$,
$p_1\equiv 0 \mod 4$ and  $w_4=0$ whose structure group reduces to a
2-torus. To produce the remaining bundles with $p_1\equiv 2 \mod 4$
and $w_4\ne 0$ we consider the direct sum of $3$ two dimensional
bundles with Euler classes $e=(1,1)$, $e=(1,0)$ and $e=(2k+1,2k)$.

In the case $w_2=(0,0)$ the 2-plane bundle with $e=(2k,2l)$ has
$p_1=4(k+l)(k-l)$ and $0\mod 4$ are the allowed values. One easily
sees that this can take on any value of $0,4,12\mod 16$, but no
value of $8 \mod 16$. This gives rise, via \pref{wep}, to $\SO(5)$
bundles with $p_1\equiv 0 \mod 4$ and hence  $w_4=0$ whose structure
group reduces to a torus. On the other hand the 5-dimensional bundle
which is the direct sum of the two dimensional bundles with
$e=(2k+1,2k)$ and $e=(1,0)$ and a trivial bundle has $p_1=4k+2$ and
hence $w_4\ne 0$. Thus in this case any $\SO(k)$ bundle with $k\ge
5$ reduces to a 2-torus.
\end{proof}

Combining the above 3 theorems, we obtain:

\begin{cor}
Every vector bundle  over $ \Sph^2\times\Sph^2 \text{ and } \CP^2 \#
- \CP^2$ with fiber dimension at least $6$ has a complete metric
with \nnc.
\end{cor}

We finally observe that the examples of non-negative curvature in
the above results can be easily obtained by direct methods, avoiding
\tref{curvature} and \tref{toruslift}. Since this argument also
 generalizes to the remaining nonnegatively curved
4-manifold $\CP^2 \#  \CP^2$, which does not admit a \coa, we
present it here.

\begin{thm}\label{general}
Any principal $\SO(k)$ bundle over $ \Sph^2\times\Sph^2 \text{ or }
\CP^2 \# \pm \CP^2$ whose structure group reduces to a torus, admits
an $\SO(k)$ invariant metric with nonnegative curvature.
\end{thm}

\begin{proof}
As was shown in \cite{Ta}, these three manifolds can be written as a
quotient of $\S^3\times\S^3$ by an action of a 2-torus. In the case
of $ \Sph^2\times\Sph^2$ this is of course simply the right action
by a maximal torus. The manifold $\CP^2 \# - \CP^2$ is the quotient
under the action $(z,w)\star (q_1,q_2) = ( zq_1 , zq_2 w )$, where
$(q_1,q_2)\in \S^3\times\S^3$ and $(z,w)\in \T^2$. For $\CP^2 \#
\CP^2$ the action $(z,w)\star (q_1,q_2) = ( zwq_1w^2 , \bar{z}wq_2
\bar{z}^2)$ suffices. Thus in all 3 cases, $\S^3\times\S^3$ is the
total space of a principal 2-torus bundle over $M$ and a biinvariant
metric is  $\T^2$ invariant. Since $H^2(M,\Z)\simeq\Z^2$, this is
the only 2-torus bundle with simply connected total space and any
other 2-torus bundle is  a quotient of $\S^3\times\S^3$ by a finite
subgroup of $\T^2$. A $\T^\ell$ bundle with $\ell>2$ is simply a
product of a 2-torus bundle with $\T^{\ell-2}$. For a circle bundle,
one observes that circle bundles with simply connected total space
can be described as $\S^3\times\S^3/\S^1_{p,q}$, where
$\S^1_{p,q}\subset\T^2$ is a circle of slope $(p,q)$ with
$\gcd(p,q)=1$. Indeed, in terms of appropriate generators $x,y$ of
$H^2(M,\Z)$ given by transgression of the natural generators in
$H^1(\T^2,\Z)$, the Euler class of such a bundle is $-qx+py$ and the
total space of a circle bundle over a simply connected base is
simply connected if and only if its Euler class is primitive. All
other circle bundles are given by a quotient of this bundle by a
finite subgroup of $\S^1$.

Thus in all cases, a $\T^\ell$ principal bundle $P$ admits an
invariant metric with nonnegative curvature. This implies that the
extension $P\times_{\T^\ell}\SO(k)$ admits an $\SO(k)$ invariant
metric with nonnegative curvature.
\end{proof}

{\it Remarks.} (a) The proof  can be applied to bundles over $\CP^n
\# \pm \CP^n$ as well. In particular any 2-plane bundle over these
manifolds admits nonnegative curvature. In the case of $\CP^n \# -
\CP^n$ this was first shown in \cite{Ya}.

 (b) One
can carry out an analysis as in \tref{product-sum} for $\CP^2 \#
\CP^2$ as well. But in this case, a circle bundle with $e=(a,b)$ has
first Pontryagin class $p_1=a^2+b^2$ and hence only bundles with
\nn\ \pont\ class are obtained.  It follows  that every $\SO(k)$
bundle with $k\ge 10$ and $p_1\ge 0$ admits a reduction to a torus
and hence a metric with \nnc.

\bigskip

We now proceed to consider the most interesting case, that of
$\SO(3)$ principal bundles over $\CP^2$, since it has the geometric
applications in Theorem A and B in the Introduction. The only case
where both singular orbits have codimension two is  the action in
\eqref{CP2}. Since we know that it admits a lift when the structure
group reduces to $\SO(2)$, we may assume that the image of $\Kpm$ in
$\SO(3)$ is not contained in an $\SO(2)$. It is then not hard to see
that the only possible group diagram as in \eqref{principal diagram}
defining an $\SO(3)$ principal bundle over $\CP^2$ is given by:

\smallskip

\begin{equation}\label{S3prCP2}
\begin{split}
         \xymatrix{
            & {\SO(3)\times \S^3} & \\
            {(R_{2,3}(p_-\gt),e^{i\gt})\cup (R_{1,3}(\pi),j)\cdot
\S^1}\ar@{->}[ur]^{} & & {(R_{1,3}(p_+\gt),e^{j\gt})}\ar@{->}[ul]_{}
\\
            &
{\Z_4=\langle(R_{1,3}(\pi),j)\rangle}\ar@{->}[ul]^{}\ar@{->}[ur]_{}
&
            }
\end{split}
\end{equation}

\no where $p_-$ is even and $\con p_+=2(4) $ in order for the
diagram to be consistent. Here and in what follows $\S^1$ will
denote the identity component of the group $\Km$, and we will use
$R_{j,k}(\gt)$  to denote the group of rotations by an angle $\gt$
in the 2-plane spanned by $j,k$.

To characterize these bundles topologically we show:

\begin{thm}\label{SO3CP2}
The principal  $\SO(3)$  bundle over  $\CP^2$ defined by
\eqref{S3prCP2} has Pontryagin class $p_1=\frac{1}{4}(p_+^2-p_-^2)$
and is spin if and only if $\con p_\pm=2(4) $. Furthermore, in the
spin case, these are precisely the principal bundles which are
obtained as pullback of $\SO(3)$ principal bundles over $\Sph^4$
under the two fold branched cover $ \CP^2\to \Sph^4$.
\end{thm}

\begin{proof}

Let us first consider the case where $\con p_-=2(4) $ and $\con
p_+=2(4) $. To see that the manifold is spin, consider  the $\S^3$
principal bundle over $\CP^2$ whose total space is the cohomogeneity
one manifold with the following group picture:

\begin{equation}\label{spinprCP2}
\begin{split}
         \xymatrix{
            & {\S^3\times \S^3} & \\
            {(e^{ip_-^*\gt},e^{i\gt})\cup (j,j)\cdot
S^1}\ar@{->}[ur]^{} & & {(e^{jp_+^*\gt},e^{j\gt})}\ar@{->}[ul]_{}
\\
            &
{\H=\Z_4}\ar@{->}[ul]^{}\ar@{->}[ur]_{} &
            }
\end{split}
\end{equation}

\no  In order for this group diagram to be consistent, we need
$p_-^*$ and $p_+^*$ odd and $\H=\langle(j,j)\rangle$ if $ \con
p_+^*=1(4) $ and $\H=\langle(-j,j)\rangle$ if $ \con p_+^*=3(4) $.
If we set  $p_\pm^*=p_\pm/2$, the group diagram \eqref{spinprCP2} is
a lift of \eqref{S3prCP2} and hence the manifold is spin.

Now  define a two fold branched cover of \eqref{spinprCP2}  onto the
cohomogeneity one manifold defined by:

\begin{equation}\label{spinprS4}
\begin{split}
         \xymatrix{
            & {\S^3\times \S^3} & \\
            {(e^{ip_-^*\gt},e^{i\gt})\cup (j,j)\cdot
S^1}\ar@{->}[ur]^{} & & {(e^{jp_+^*\gt},e^{j\gt})\cup (i,i)\cdot
S^1}\ar@{->}[ul]_{}
\\
            &
{\Delta Q}\ar@{->}[ul]^{}\ar@{->}[ur]_{} &
            }
\end{split}
\end{equation}

\no where we need $p_\pm^*$ to be odd for the group diagram to be
consistent, as long as we allow a sign change in some of the
components of $\Delta Q$. Notice that \eqref{spinprS4}  is an $\S^3$
principal bundle over $\Sph^4$ since after dividing by $\S^3\times
1$ we obtain the cohomogeneity one picture \eqref{S^4} for $\Sph^4$.
This two fold branched cover  from \eqref{spinprCP2} to
\eqref{spinprS4}  becomes the two fold branched cover $ \CP^2\to
\Sph^4$ after dividing by $\S^3\times 1$. In \cite{GZ} it was shown
that every  $\S^3$ principal bundle over $\Sph^4$ is of the form
\eqref{spinprS4} and has Euler class
$e=\frac{1}{8}((p_+^*)^2-(p_-^*)^2)$ and hence Pontryagin class
$p_1=\frac{1}{2}((p_+^*)^2-(p_-^*)^2)$. Since furthermore $ \CP^2\to
\Sph^4$ induces multiplication by two in dimension 4, it follows
that $p_1=\frac{1}{4}(p_+^2-p_-^2)$.

\smallskip

We now consider the manifolds $P$ described by  \eqref{S3prCP2}
where $\con p_-=0(4) $ and $\con p_+=2(4) $. We claim that these are
 $\SO(3)$ principal bundles over $\CP^2$ which are not spin. By
\pref{HP}, this amounts to showing that $P$ is simply connected. We
can determine $\pi_1(P)$ by applying van Kampen to the decomposition
$P = P_- \cup P_+$, $P_- \cap P_+ = P_0$ provided by its
cohomogeneity one description. Clearly $P_\pm$ is homotopy
equivalent to $\L \times \G/\Kpm$ and $P_0$ to $\L\times \G/\H$,
where $\L = \SO(3)$ and $\G = \S^3$ according to our recipe. We now
need to compute
  $\pi_1(\L\times \G/\Kpm)$ and their inclusion into $\pi_1(\L\times
  \G/\H)$. This is easiest done if we write the orbits as quotients of
  $\S^3\times \S^3$ since then the fundamental group is
   isomorphic to the group of
  components of the isotropy group. Now observe that for the
  preimage of $\Kp$ into $\S^3\times \S^3$, the component group can be
  represented
  by $(-1,1)$ since $\con p_+=2(4) $. For the group $\Km$,
   $\con p_-=0(4) $ implies that
    $(-1,1)\in \Kmo$ and hence
   $\pm (j,j)$ and $\pm (1,1)$ represent the component group.
   Finally, for $\H$ its preimage  is generated by $(j,j)$
   and $(-1,1)$. Altogether, we conclude that $P$ is simply
   connected and thus $w_2\ne 0$.

To compute the integer $p_1(P)$ we use \pref{HP}. Since $w_2\ne 0$
\lref{wep} implies that $\con p_1(P)=1(4)$. In particular,
$p_1(P)\ne 0$,  and by \pref{HP}, $H^4(P,\Z)$ is finite with $|
p_1(P)|= |H^4(P,\Z)|$.  In order to compute $|H^4(P,\Z)|$, we apply
the Mayer Vietoris sequence to the decomposition $P = P_- \cup P_+$.
In particular, we need the cohomology groups of the pieces:

\begin{lem}\label{orbits} The cohomology groups of $P_- , P_+$ and
$P_0$ satisfy:

\begin{itemize}
\item[(a)] $P_-$ has
$H^3(P_-) = \Z \oplus \Z_2$ and $H^4(P_-) =0$.

\item[(b)] $P_+\simeq \Sph^2\times \SO(3)$ and hence
  $H^3(P_+) = \Z$ and
$H^4(P_+) = \Z_2$.

\item[(c)] $P_0\simeq (\Sph^3/\Z_4)\times \SO(3)$ and hence
 $H^3(P_0) = \Z \oplus \Z \oplus
\Z_2$ and $ H^4(P_0) = \Z_2$.
\end{itemize}
\end{lem}

\begin{proof}
Let $B_{\pm}=\G/\Kpm\subset \CP^2$. First observe that the
restriction maps $H^2(\CP^2,\Z_2)\to H^2(B_\pm,\Z_2)$ are 0 in the
case of $B_+=\Sph^2$, and an isomorphism in the case of $B_-=\RP^2$.
Indeed, this follows from the Mayer-Vietoris sequence for the
decomposition $M = M_- \cup M_+$:

$$0=H^1(M_0)  \to H^2(M)\to  H^{2}(M_-)\oplus H^{2}(M_+)
      \to H^{2}(M_0) \to H^3(M)=0 . $$

\no By considering this sequence first over the integers and using
$H^2(M_-)=H^2(\RP^2)=0 , H^2(M_0)=H^2(\S^3/\Z_4)=\Z_4$, it follows
that $H^2(M,\Z)\to H^2(M_+,\Z)$ is multiplication by 4, and hence
mod 2 becomes the 0 map. We then consider this sequence with $\Z_2$
coefficients to conclude that $H^2(M,\Z_2)\to H^2(M_-,\Z_2)$ is an
isomorphism.

We next consider $P_\pm$ as an $\SO(3)$ bundle over $M_\pm$ which is
homotopy equivalent to the restriction of this bundle to $B_\pm$.
But a principal $\SO(3)$ bundle over a 2 complex is classified by
$w_2$. Hence our previous remark implies that $P_+\to B_+$ and
$P_0\to M_0$  are  trivial bundles and thus  $P_+\simeq \Sph^2\times
\SO(3)$ and $P_0\simeq \S^3/\Z_4\times \SO(3)$. This determines
their cohomology groups.

For $P_-\to B_-=\RP^2$, it follows that $w_2=1$, and hence it is the
unique non-trivial principal $\SO(3)$ bundles over $\RP^2$. We can
therefore use the following more convenient description for $P_-$:
Consider $\Sph^2\times \S^3/\Z_4$, where $\Z_4$ acts via the
antipodal map on $\Sph^2$ and via left multiplication of $\Z_4$ on
the unit quaternions $\S^3$. This can also be described as
$\Sph^2\times \SO(3)/\Z_2$, which via the projection on the first
factor becomes a principal $\SO(3)$ bundles over $\RP^2$. Since the
total space has fundamental group $\Z_4$, it must be the nontrivial
principal bundle. Moreover, it can be viewed as a (nonorientable)
$\Sph^2$ bundle over the lens space $\S^3/\Z_4$, via projection onto
the second factor. The Euler class of that bundle is easily seen to
be 0 mod 2, and hence one computes, by using the Gysin sequence, the
cohomology of $\Sph^2\times \S^3/\Z_4$ with $\Z_2$ coefficients to
be $\Z_2$ in dimension $0,1,4$ and $5$ and $\Z_2\oplus\Z_2$ in
dimension 2 and 3. This only leaves the following possibilities for
the cohomology with integer coefficients: $H^3(P_-) = \Z \oplus
\Z_2$ and $H^4(P_-) =0$.

\end{proof}

We now incorporate this information into the Mayer-Vietoris
sequence. Denote by $\pi_\pm\colon P_0 \to P_\pm$ the projections of
the sphere bundle $P_0\cong \L\times \G/\H \to \L\times \G/\Kpm\cong
P_\pm$. We then have:
\begin{align*}
0=H^3(P)  &\to  H^{3}(P_-)\oplus H^{3}(P_+)=\Z\oplus\Z_2\oplus\Z
\overset{\pi_-^* - \pi_+^*}{\longrightarrow}
      H^{3}(P_0)= \Z\oplus\Z\oplus\Z_2
      \to H^4(P)  \to \\
     &\to H^{4}(P_-)\oplus
H^{4}(P_+)= \Z_2 \to H^{4}(P_0)=\Z_2\to\cdots
\end{align*}
where  we have used the fact from \eqref{HP} that $H^3(P) = 0$.
Using the Gysin sequence of the circle bundle $P_0\to P_+$, this
also implies that $ H^{4}(P_+)\to H^{4}(P_0)$ is an isomorphism. It
follows that $ H^4(P) $ is a finite group whose order is equal to
the order of the cokernel of the map $\pi_-^* - \pi_+^*$ restricted
to $\Z^2$.

To compute the order of this cokernel, we proceed as in
    \cite{GZ}, diagram (3.5) and (3.6), and use the same notation.
    The projection
    $\eta\colon \L\times \G\to \L\times \G/\H=(\Sph^3/\Z_4)\times \SO(3)$
    is, as in that case, an 8-fold cover and
    induces a map
    with determinant 8 on $\Z\oplus\Z$. But the
    projection $ \mu_+\colon \L\times \G/\Kpo=\Sph^2\times \S^3\to
    \L\times \G/\Kp=\Sph^2\times \SO(3)$ is now a two fold cover which
    induces multiplication by 2 on $H^3=\Z$. For the projection
    $ \mu_-\colon \L\times \G/\Kmo\to
    \L\times \G/\Km$, which is now a 4-fold cover, we can use the
    description  obtained in the proof of \lref{orbits}, which
    implies
that this cover can be described as
    $\Sph^2\times \S^3\to \Sph^2\times \S^3/\Z_4$ which induces
     multiplication by
    4 on $H^3=\Z$. Finally, observe that the preimage of the circle
    $(R_{1,3}(p_+\gt),e^{j\gt})\subset \SO(3)\times \S^3$ is
    $(e^{j\frac{p_+}{2}\gt},e^{j\gt})\subset \S^3\times \S^3$
    since $p_+$ is even,
     and similarly
    for $\L\times \G/\Km$.

Thus we get $$|p_1|=|H^4(P)|=|\text{cokernel}(\pi_-^* - \pi_+^*)|=
\frac{1}{8}\det \left(\begin{array}{rr} -4&2
\\
p_-^2 & -\frac{1}{2}p_+^2
\end{array}\right) = \frac{1}{4}|p_+^2-p_-^2|.$$
Since $\con p_1(P)=1(4) $, $\con p_-=0(4) $ and $\con p_+=2(4) $ it
follows  that $p_1=\frac{1}{4}(p_+^2-p_-^2)$.  This completes the
proof of \ref{SO3CP2}.
\end{proof}

We can now determine which bundles admit a lift of the action. We
leave out the case where the structure group reduces to a torus,
since we already know that it always admits a lift in that case.

\begin{cor}\label{lifts}
Let $P\to \CP^2$ be a principal $\SO(k)$ bundle whose structure
group does not reduce to a torus and consider the cohomogeneity one
action of $\SO(3)$ on $\CP^2$.
\begin{itemize}
\item[(a)] For $k=3$, the action has a lift unless $w_2=0$ and
 $p_1 \equiv 4$ mod $8$.
\item[(b)] For $k=4$,  the  action has a lift unless $w_2=0$
 and $p_1(P)=2k+2l\, ,\, e(P)=k-l$ and $k,l$ not
both even.
\item[(c)] For $k\ge 5$, the action admits a lift to every
principal bundle.
\end{itemize}
\end{cor}

\begin{proof}
In the case of $k=3$ and $w_2\ne 0$ we  need to show, due to
\lref{wep}, that all values of $\con p_1(P)=1(4)$ are assumed. But
one easily sees that if $p_-=4s$ and $p_+=4r+2$, then
 $p_1=4(r^2-s^2+r)+1$  achieves all values of 1 mod
4. Similarly, if $w_2=0$, all values of $0$ mod $8$ are assumed. In
the case of $k=4$ the result follows from \tref{SO3CP2} by applying
\lref{cover}(c), \pref{so4}, \pref{prinso4} and \eqref{ppm}.

For $k\ge 5$ we use the general fact that a $k$ dimensional vector
bundle over $M^4$ is the direct sum of a $4$ dimensional vector
bundle with a trivial one, i.e. the $\SO(k)$ principal bundle can be
viewed as an extension of an $\SO(4)$ bundle. Thus, in the case of
$w_2\ne 0$, \lref{reductionlift} (a) implies that every $\SO(k)$
principal bundle admits a lift.

If $w_2=0$ we need to show, due to \pref{prinsok}, that $p_1$ can
achieve every even value. From the case of $k=4$, it follows that
every value of $0$ mod 4 is assumed. We can now take the direct sum
of a 3 dimensional vector bundle with $w_2\ne 0$ and $p_1=4s+1$ with
a two dimensional vector bundle with Euler class one and hence
$w_2\ne 0$ and $p_1=1$, which admits a lift by Lemma \ref{cover} (c)
since both do. By the product formula for Pontryagin classes and
Stiefel Whitney classes, we have $w_2=0$ and $p_1=4s+2$. Thus every
even value of $p_1$ is already assumed for $5$ dimensional vector
bundles. As we saw in Table A, the value of $w_4$ is determined by
whether $p_1$ is $2$ or $0$ mod $4$.
\end{proof}

\smallskip

\tref{SO3CP2} and \cref{lifts} implies Theorem C as well as Theorem
A in the Introduction.   For complex vector bundles we can do better
since in that case $P^+$ reduces to $\SO(2)$ which always has a
commuting lift.
 In the case of $w_2=0$, the bundles with
 $p_1(P^-)\equiv 0$ mod $8$ have a lift and
 Theorem B follows since by \eqref{chern}
$p_1(P^-)=c_1^2-4c_2$ and furthermore one has the general fact that
$w_2 = c_1$ mod $2$.

\section{Lifts of sum actions}
In this section we will determine which bundles admit commuting
lifts of the cohomogeneity one actions on 1-connected 4-manifolds
where at most one singular orbit is of codimension two. Up
to extensions there is one such action on $\CP^2$ and two on $\Sph^4$.

\smallskip

We begin with the sum action of $\SO(2) \SO(3)$ on $\Sph^4$, where the satisfactory answer is:

\begin{thm}
The sum action of $\SO(2) \SO(3)$ on $\Sph^4$ admits a commuting
lift to every principal $\SO(k)$  bundle.
\end{thm}

\begin{proof}
As we saw earlier, it suffices to prove the claim for principal
$\S^3$ bundles. Consider the  Brieskorn variety  $M_d^7$ defined by
the equations
$$ z_0^d + z_1^2 +\cdots z_4^2=0 \quad , \quad |z_0|^2 + \cdots |z_4|^2 =1 .$$

\no It  carries
 a \co\ one action by $\SO(2)\SO(4)$
defined by (cf. \cite{HHs})
$$(e^{i\gt},A)(z_0,\cdots ,
z_4)=(e^{2i\gt}z_0,e^{id\gt}A(z_1,\cdots,z_4)^t)$$

\no The isotropy groups are given by (cf.  \cite{BH}): $$ \Km=\SO(2)
\SO(2) \, ,\, \Kp= \O(3) \text{ and }  \H=\Z_2\times \SO(2)$$

\no  The normal subgroup $\SU(2)\subset \SO(4)$ acts freely on
$\R^4$ as left multiplication by quaternions, hence freely on $\C^4$
and thus on $M_d^7$ as well. The quotient is a 4-manifold with an
induced \co\ one action by $\SO(2)\SO(3)$ with $ \Km=\SO(2) \SO(2)$,
$ \Kp= \SO(3)$ (effectively) and $\H=\SO(2)$ and hence must be the
\co\ one action \eqref{sum2} on $\Sph^4$. Thus $M_d^7$ is a
principal $\SU(2)$ bundle over $\Sph^4$ for which the sum action has
a lift. To see which bundle it is, recall that such bundles are
classified by their Euler class. Furthermore, it follows from the
Gysin sequence that the Euler class (evaluated on a fundamental
class) is the order of the fourth cohomology group of the total
space. For $M_d$ we have (cf. \cite[p.275]{Br}) $H^4(M_d,\Z)=\Z_d$
and hence it is the bundle with Euler class $d$. Thus the  action
\eqref{sum2} on $\Sph^4$ lifts to every principal $\SU(2)$ bundle
over $\Sph^4$.
\end{proof}

\bigskip

The remaining actions, both have fixed points. Here we start with the \co\ one action on $\CP^2$ given by
 the standard $\SU(2)$ and $\U(2)$  action
on $\CP^2$ with a fixed point, see \eqref{cp2fix}. By
\eqref{reduction} (b) it suffices to consider the $\SU(2)$ action.

A lift of the cohomogeneity one diagram \eqref{cp2fix}  on $\CP^2$
to a principal $\SO(k)$ bundle must be given by a diagram::

\begin{equation}\label{Fixprin}
\begin{split}
         \xymatrix{
            & {\SO(k)\times \S^3} & \\
            {\S^3}\ar@{->}[ur]^{(\phi_-,j_-)} & &
{\S^1}\ar@{->}[ul]_{(\phi_+,j_+)} \\
            & {1}\ar@{->}[ul]\ar@{->}[ur] &
            }
\end{split}
\end{equation}

\no where $j_-=\id$ and $j_+$ is the  inclusion into any fixed
circle subgroup of $\S^3$. The group diagram is hence determined by
the homomorphisms $\phi_-\colon \S^3\to \SO(k)$ and $\phi_+\colon
\S^1\to\SO(k)$. By \lref{equiv}, we can conjugate these
homomorphisms $\phi_-$ and $\phi_+$ separately into a normal form.
It is well known that $\phi_-$ is, up to equivalence, given by a sum
of irreducible representations $\phi_1 + \cdots +\phi_r$ where
$\phi_i$ either has dimension $2n_i+1$ (allowing $n_i = 0$) or
$4n_i$. By a theorem of Malcev \cite{Ma} the image group
$\phi_-(\S^3)\subset\SO(k)$ is unique up to conjugacy, unless $k$ is
even. In that case one has an outer automorphism $A$ and
$\phi_-(\S^3)$ and $A(\phi_-(\S^3))$ are conjugate in $\SO(k)$
unless the irreducible sub-representations are all non-trivial even
dimensional. Although this change will give a \com\ which, by
\lref{reduction}, is not equivariantly diffeomorphic to the original
one, we will see that the corresponding principal bundles are
isomorphic if $k=3$ or $k\ge 5$. We define:
\begin{align*}
m_i&=n_i(2n_i+1)(2n_i+2)/3 \; \text{ if }  \dim \phi_i=2n_i+1 \\
m_i&=(2n_i-1)2n_i(2n_i+1)/3 \; \text{ if } \dim \phi_i=4n_i
\end{align*}
 Furthermore, $\phi_+$ is given by $e^{i\theta}\to \diag(
R(q_1\theta),\dots , R(q_r\theta))$ if $k$ even and if $k$ is odd
$e^{i\theta}\to \diag( R(q_1\theta),\dots , R(q_r\theta),1)$ , where
$r=[k/2]$ and $(q_1,\dots ,q_r)$ are relatively prime integers. We
can now state our classification theorem of these principal bundles
as follows:

\begin{thm}\label{Fix}
Let $P$ be the principal  $\SO(k)$  bundle over  $\CP^2$ defined by
\eqref{Fixprin}  for some integers $n_i$ and $q_i$. Assuming that
$k=3$ or $k\ge 5$, we have  $p_1=\pm (\sum q_i^2-\sum m_i)$ and $w_2
\equiv \sum q_i$ mod $2$.
\end{thm}

\begin{proof}
We again use  the decomposition $P=P_-\cup P_+$ with $P_-\simeq
\SO(k)\times \S^3/\S^3 = \SO(k)$ and $P_+\simeq \SO(k)\times
\S^3/\S^1$ and $P_0=P_-\cap \P_+\simeq \SO(k)\times \S^3$ in order
to apply \eqref{HP}. Notice that $\phi_+$ is onto in $\pi_1$, and
hence $P_+$ simply connected, if and only if $\sum q_i$ is odd.
Since $\pi_1(P_-)=\Z_2\to \pi_1(P_0)=\Z_2$ is an isomorphism, van
Kampen implies that $P$ simply connected, which by \pref{HP} means
that $w_2\ne 0$, if and only if $\sum q_i$ is odd.

For the cohomology of the principal orbits we have
$H^3(P_0)=H^3(\SO(k)\times \S^3)=\Z\oplus\Z$. We choose a generator
$x\in H^3(\SO(k))$ and $y\in H^3(\S^3)$ and by abuse of notation use
the same symbol for a basis in $H^3(P_0)$. The sign of these
generators will be determined uniquely in the proof of
\lref{orbit+}. Let us first assume that $k\ge 5$.

\begin{lem}\label{orbit+} For $P_+$ we have
$H^3(P_+)=\Z$ and
\begin{itemize}
\item[(a)] If $\sum q_i$ is odd, $H^4(P_+) = \Z_2$ and under the
projection $\pi_+^*\colon H^3(P_+)\to H^3(P_0)$ a generator goes to
$-2x+(\sum q_i^2)y$.
\item[(b)] If $\sum q_i$ is even, $H^4(P_+) = \Z_2\oplus\Z_2$
and $\pi_+^*$ takes a generator to $-x+\frac 1 2 (\sum
q_i^2)y$.
\end{itemize}
\end{lem}

\begin{proof}
Let us first recall the Borel method of computing the cohomology of
a homogeneous space $\G/\K$.  Let $E$ be a space on which $\G$ acts
freely, and hence  $B_{\G}=E/\G$ and $B_{\K}=E/\K$  the classifying
spaces for principal $\G$ and $\K$ bundles respectively. One uses
the naturality between the differentials in the following
commutative diagram of $\G$ principal fibrations:

\begin{picture}(100,160)\label{DA}
\put(145,130){$\G$} \put(257,130){$\G$}
\put(180,133){\vector(1,0){55} } \put(145,117){ \vector(0,-1){20} }
\put(258,117){ \vector(0,-1){20} } \put(123,80){$ \G \times_{\K} E$
} \put(250,80){$\G\times_{\G}E=E$} \put(180,83){\vector(1,0){55} }
\put(200,72){$\varphi$} \put(145,65){  \vector(0,-1){20} }
\put(150,55){ {$\pi$ } } \put(258,65){ \vector(0,-1){20} }
\put(265,55){ {$\tilde{\pi}$ } } \put(144,30){$B_{\K}$}
\put(258,30){$B_{\G}$} \put(180,33){\vector(1,0){55} }
\put(200,22){$B_{f}$} \put(180,-3){\it \small Diagram A}
\end{picture}

\vspace{20pt}

\no where $f\colon \K\to\G$ is the inclusion. The right hand side
fibration is the universal $\G$ principal bundle. In the left hand
side fibration $\G$ acts freely on $ \G \times_{K} E$ via left
multiplication in the first coordinate and $\pi$ is the projection
onto the second coordinate. The map $B_{f}$ is therefore the
classifying map of this principal bundle. The spectral sequence for
the left hand side fibration computes the cohomology of $\G/\K$
since the projection onto the first coordinate $\G \times_{\K}
E\to\G/\K$ is a homotopy equivalence.  The differential in the
spectral sequence are thus determined as soon as one computes
$B_f^*$. In order to compute this map, one uses a further
commutative diagram:

\vspace{-50pt}

\begin{picture}(100,160)\label{DB}
\put(143,80){$ B_{\K}$ } \put(255,80){$B_{\G} $}
\put(180,83){\vector(1,0){55} } \put(200,72){$B_f$} \put(145,50){
\vector(0,1){20} } \put(150,55){ {$B_{g'}$ } } \put(258,50){
\vector(0,1){20} } \put(265,55){ {$B_{g} $ } }
\put(144,30){$B_{\T'}$} \put(258,30){$B_{\T}$}
\put(180,33){\vector(1,0){55} } \put(200,22){$B_{f}$}
\put(180,-3){\it \small Diagram B}
\end{picture}

\bigskip

\no where $g'\colon \T'\to \K$ and $g\colon \T\to \G$ are maximal
tori.  The cohomology $H^*(B_{\T})=\Z[x_1,\dots,x_r]$ is a
polynomial ring with $\dim x_i=2$ and $r=\rank \G$. This method
works well if the Lie groups involved have no torsion in cohomology
since then $B_{g}^*$ is injective with image the Weyl group
invariant elements. Extra care needs to be taken since this is not
true for $\SO(n)$.

We now apply this to our situation where $\G=\G_1\times \G_2
=\SO(k)\times \S^3$, $\K=\Kp=\T'=\S^1$ and
$f=(f_1,f_2)=(\phi_+,j_+)$. By using the naturality of differentials
with respect  to the projection $\G_1\times\G_2\to \G_i$ we can
break up the computation of the differentials in the left hand side
spectral sequence into considering two diagrams of type A, one for
$\G=\G_1$ and one for $\G=\G_2$. We start with the former one.

As in the proof of \pref{HP}, we see that in the universal bundle
for $\G_1$ we have that $\tilde{d}_2\colon \tilde{E}_2^{0,3} =
H^3(\G_1,\Z)=\Z \to \tilde{E}_2^{2,2}=H^2(B_{\G_1},H^2(\G_1))= \Z_2$
is onto  with image the second Stiefel Whitney class.  Furthermore
$\tilde{d}_4\colon \tilde{E}_4^{0,3}=\Z\to \tilde{E}_4^{4,0}=\Z$
takes the generator $2x$ to a generator which we denote by
$\bar{x}$.

In order to compute $B_{f_1}^*$ let $g_1\colon \T^r\to \SO(k)$
 be a maximal torus with coordinates
$(s_1,\dots ,s_r)$. By abuse of notation we identify $s_i\in
H^1(T^r)$ and via transgression $\bar{s}_i\in H^2(B_{\T^r})$ and
hence $H^2(B_{\T^r})=\Z[\bar{s}_s,\dots , \bar{s}_r]$. Similarly $u$
is a coordinate in $\K=\S^1$ and hence
 $H^2(B_{\S^1})=\Z[\bar{u}\, ]$. Next we
claim that $B_{g_1}^*(\bar{x})=\sum\bar{s_i}^2$. To see this, let
$\SO(2)$ and $\SO(3)$ be the standard embeddings in $\SO(k)$. Since
the Stiefel manifold $\SO(k)/\SO(3)$ is 2-connected with
$\pi_3=\Z_2$, it has cohomology $H^1=H^2=H^3=0$ and $H^4=\Z_2$ and
the spectral sequences of the bundle $\SO(k)/\SO(3)\to B_{\SO(3)}\to
B_{\SO(k)}$ implies that $H^4(B_{\SO(k)})=\Z\to H^4(B_{\SO(3)})=\Z$
is an isomorphism. Furthermore, the spectral sequence for
$\SO(3)/\SO(2)\to B_{\SO(2)}\to B_{\SO(3)}$ shows that
$H^4(B_{\SO(3)})\to H^4(B_{\SO(2)})=\Z$ is an isomorphism as well.
We now choose the sign of the generator $x$ and thus $\bar{x}$, so
that the embedding $\SO(2)\to\SO(k)$ takes $\bar{x}$ to the square
of a generator  in $H^2(B_{\SO(2)})=\Z$. This easily implies that
$B_{g_1}^*(\bar{x})=\sum\bar{s_i}^2$. Since $f_1(u)=(q_1u,\dots ,
q_ru)\in\T^r$ we have $B_{f_1}^*(\bar{s_i})=q_i\bar{u}$ and hence
$B_{f_1}^*(\bar{x})=\sum q_i^2 \bar{u}^2$. By naturality it follows
that $d_4(2x)=\sum q_i^2 \bar{u}^2$. Next we claim that $d_2(x)\in
{E}_2^{2,2}=H^2(B_{\S_1},H^2(\G_1))= \Z_2$ is non-zero if and only
if $\sum q_i$ is odd. This indeed follows since by \pref{HP} the
$\SO(k)$ principal bundle $\SO(k)/f_1(\S^1)=\G_1\times_{\S^1}E\to
B_{\S^1}$ is  spin if and only if the total space $\SO(k)/f_1(\S^1)$
is not simply connected.

In the universal bundle for $\G_2$ we have that $\tilde{d}_4\colon
\tilde{E}_4^{0,3} =H^3(\G_2)= \Z \to
\tilde{E}_4^{4,0}=H^4(B_{\G_2})= \Z$ is an isomorphism and we denote
the image of the generator $y$ by $\bar{y}$. $f_2\colon \S^1\to
\S^3$ can be viewed as a maximal torus and from the spectral
sequence of $\S^3/\S^1\to B_{\S^1}\to B_{\S^3}$ it follows that we
can choose the sign of $y$ such that $B_{f_2}^*(\bar{y})=\bar{u}^2$
and hence $d_4(y)=\bar{u}^2$.

We are now ready to consider the spectral sequence of the $\G$
principal bundle $\SO(k)\times \S^3/\S^1 =\G\times_{\S^1}E\to
B_{\S^1}$. The group $E_2^{0,3}=H^3(\G)=\Z\oplus \Z$ is generated by
$x,y$. Assuming that $\sum q_i$ is odd, we obtain by the above that
$d_2(x)\ne 0$ and $d_2(y)=0$. Hence $2x,y$ are generators of
$E_3^{0,3}=\Z\oplus\Z$  and $d_4(2x)=\sum q_i^2\bar{u}^2$ and
$d_4(y)=\bar{u}^2$. Thus $E_{\infty}^{4,0}=E_{\infty}^{2,2}=0$ and
$E_{\infty}^{0,4}=\Z_2$ and hence $H^4(P_+)=\Z_2$. Furthermore,
$E_{\infty}^{0,3}=\Z$ with generator $-2x+\sum q_i^2 y$. This
implies that $H^3(P_+)=\Z$ and via the edge homomorphism  $-2x+\sum
q_i^2 y$ is the image of a generator in $H^3(\G/\S^1)$.

If on the other hand $\sum q_i$ is even, $d_2(x)=d_2(y)=0$ and hence
$H^4(P_+)=\Z_2\oplus\Z_2$. Now $x,y$ generates $E_3^{0,3}$  and
$d_4(x)=\frac 1 2 \sum q_i^2\bar{u}^2$ and $d_4(y)=\bar{u}^2$ and
hence $-x+\frac 1 2 \sum q_i^2 y$ is the image of a generator. This
completes the proof of \ref{orbit+}.
\end{proof}

For the left half we have:

\begin{lem}\label{orbit-}  $P_-$ satisfies:
\begin{itemize}
\item[(a)]
  $H^3(P_-) = \Z$ and $H^4(P_-) = \Z_2$ and $H^4(P_-)\to H^4(P_0)=\Z_2$ is an isomorphism.
\item[(b)] Under the projection $\pi_-^*\colon H^3(P_-)\to H^3(P_0)$
a generator goes to $-x+\frac 1 2 (\sum m_i^2)y$.
\end{itemize}
\end{lem}

\begin{proof}
We indicate the changes which are necessary. We now have
$\K=\Km=\S^3$ and $f=(f_1,f_2)=(\phi_-,j_-)$. If $g'\colon
\S^1\to\S^3$ is a maximal torus of $\K$ with coordinate $u$, we can
identify a generator in $H^4(B_{\K})=\Z$ with $\bar{u}^2$ via
$B_{g'}$. In the spectral sequence of $\G\times_{\S^3}E\to B_{\S^3}$
we have that $d_2(x)\in E_2^{2,2}=0$ vanishes and hence $x,y$
generate $E_3^{0,3}$ with $d_4(y)=\bar{u}^2$. It remains to compute
$d_4(x)$. For this purpose, let us assume momentarily that $f_1$ is
an irreducible representation of dimension $2n+1$. Standard
representation theory implies that the homomorphism $f_1\colon
\S^1\to \T^r\subset\SO(k)$ takes $e^{i\theta}\in\S^1$ to
$\diag(R(2n\theta), R((2n-2)\theta),\dots,R(2\theta))\in\T^r$. Hence
$B_{f_1}^*(\bar{x})=B_{f_1}^*(\sum \bar{s_i}^2)=2^2+4^2+\dots
+(2n)^2=n(2n+1)(2n+2)/3$. If on the other hand $f_1$ is an
irreducible representation of dimension $4n$, it can be viewed,
using the usual embedding $\SU(2n)\subset\SO(4n)$, as a complex $2n$
dimensional representation. Thus
$f_1(e^{i\theta})=\diag(R((2n-1)\theta)
,\dots,R(\theta),R((2n-1)\theta),\dots ,R(\theta))\in\T^{2n}$.
Notice that if we change the representation by an outer
automorphism, which means we change the embedding
$\SU(2n)\subset\SO(4n)$,  the induced homomorphism on the maximal
torus is the same. Thus $B_{f_1}^*(\bar{x})=2(1^2+3^3+\dots +
(2n-1)^2)=(2n-1)2n(2n+1)/3 $. This process is clearly additive and
we obtain $B_{f_1}^*(\bar{x})=\sum m_i \bar{u}^2$ and thus
$d_4(x)=\frac 1 2 \sum m_i \bar{u}^2$. As above, this finishes the
proof of \eqref{orbit-}.
\end{proof}

We are now ready to combine the information in \lref{orbit+} and
\lref{orbit-} in the Mayer Vietoris sequence of $P=P_-\cup P_+$.
Still assuming $k\ge 5$, we have:
\begin{align*}
0=H^3(P)  &\to  H^{3}(P_-)\oplus H^{3}(P_+)=\Z\oplus\Z
\overset{\pi_-^* - \pi_+^*}{\longrightarrow}
      H^{3}(P_0)= \Z\oplus\Z
      \to H^4(P)  \to \\
     &\to H^{4}(P_-)\oplus
H^{4}(P_+)= H^{4}(P_-)\oplus\Z_2 \to H^{4}(P_0)=\Z_2\to \cdots
\end{align*}

\no From \pref{HP} we know that $|p_1(P)|=\frac 1 2 |H^4(P)|$ and
from \lref{orbit-} that $H^4(P_+)\to H^4(P_0)$ is an isomorphism. In
the case of $w_2\ne 0$, we have $H^{4}(P_-)=\Z_2$ and thus
$|H^4(P)|= 2|\text{cokernel}(\pi_-^* - \pi_+^*)|=
 2|\det
\left(\begin{array}{rr} -1& 2
\\
   \frac 1 2 \sum m_i& -\sum  q_i^2
\end{array}\right)|
=2|\sum q_i^2-\sum m_i|$.  If on the other hand $w_2=0 $,
we have $H^{4}(P_-)=\Z_2\oplus\Z_2$ and thus $|H^4(P)|=
4|\text{cokernel}(\pi_-^* - \pi_+^*)|=
4\det |\left(\begin{array}{rr}
-1& 1
\\
 \frac 1 2 \sum m_i & -\frac 1 2 \sum q_i^2
\end{array}\right)| =2|\sum q_i^2-\sum m_i|$.



Finally, if $k=3$, what changes is that $H^4(P_-)=0$ and, by
Poincare duality, $ H^4(P_+)=0$ if $\sum q_i$ odd or $\Z_2$ if $\sum
q_i$ even. Since by
 \pref{HP} we  now have $|p_1(P)|=
|H^4(P)|$, the conclusion remains the same. This completes the proof of Theorem \ref{Fix}
\end{proof}

For $k=3,4$, Theorem \ref{Fix} together with Proposition \ref{so4}
and Lemma \ref{cover} (c)  determines which principal $\SO(k)$
bundles admit lifts. As remarked earlier, we were not able to
determine the sign of $p_1$ when $k\ge 5$, which leaves an ambiguity
in our classification in this case.

 Since the 3-dimensional
representation of $\SU(2)$ has $m=4$ we have in particular:

\begin{cor}\label{fixlift}
 The $\SU(2)$ action on $\CP^2$ with a fixed point has a
lift to a principal $\SO(3)$ bundle if and only if  $p_1=q^2$ or
$p_1=\pm (q^2-4)$ for some integer $q$.
\end{cor}

Two of these bundles are well known in positive curvature \cite{Sh}:
$q=1$, $p_1(P)=-3$ is the Aloff-Wallach space  $P= \SU(3)/Z(\U(2))$
and $q=3$, $p_1(P)=5$ the Eschenburg space  $P=
\diag(z,z,z^2)\backslash \SU(3)/ \diag(1,1,z^4)$.

\bigskip

The same methods can be applied to the suspension action of $\SU(2)$
on $\Sph^4$, which has two fixed points. A lift of this action to a
\coo\ action on a principal $\SO(k)$ bundle has a group diagram as
in \eqref{Fixprin} such that $\Kp=\S^3$ as well. Hence $\phi_-$ and
$\phi_+$ are both the direct sum of irreducible representations of
$\SU(2)$ of dimensions $2n_i^-+1$ or $4n_i^-$ respectively
$2n_i^++1$ or $4n_i^+$. The principal bundles in this case are of
course all spin, i.e. they are classified by $p_1$. If $k>5$, Table
A implies that $w_4$ is determined by $p_1$ as well. Combining
\pref{HP} with \lref{orbit-} and the Mayer Vietoris sequence, we
obtain:

\begin{thm}\label{suspension}
Let $P$ be the principal  $\SO(k)$  bundle over  $\Sph^4$ defined by
the integers $n_i^-$ and $n_i^+$. Assuming that $k=3$ or $k\ge 5$,
we have  $p_1=\pm ( \sum m_i^- - \sum m_i^+ )$.
\end{thm}

For bundles over $\Sph^4$, we can reverse the sign of $p_1$ by
considering the pull back bundle under the antipodal map since it
reverses orientation. Since it also commutes with the action of
$\SU(2)$, it follows that if one bundle admits a lift, so does the
other.
 This completely
determines when the sum action of $\SU(2)$ on $\Sph^4$ admits a lift
to a principal $\SO(k)$ bundle $P$. But notice that for each fixed
$k$, there are only finitely many bundles that do.  By
\eqref{reduction} (b), the  groups $G= \U(2)$ and   $G=\SO(4)$ admit
a lift if and only if $\SU(2)\subset G$ does.

\smallskip

In \cite{HH} one also finds a classification when the action of
$\G=\SO(4)$ on $\Sph^4$ admits a lift to a principal $\SO(k)$ bundle
with $k=3$ and $k=4$. It is interesting to note that for $k=4$ the
isomorphism type of the bundle depends on the outer automorphism
group of $\SO(4)$. E.g. if $\phi_-$ and $\phi_+$ are both the
standard representation of $\Kpm\simeq\SO(4)$ on $\R^4$ the bundle
is trivial, whereas if one changes one of these by an outer
automorphism, one obtains the tangent bundle of $\Sph^4$.
\bigskip

\providecommand{\bysame}{\leavevmode\hbox to3em{\hrulefill}\thinspace}

\end{document}